\documentclass[11pt, oneside, reqno]{amsart}
\usepackage{amsfonts, amssymb, color}
\usepackage{amsmath}
\usepackage{graphicx}
\usepackage{amscd}
\usepackage{amsthm}
\usepackage{tikz}
\usetikzlibrary{decorations.markings}
\usetikzlibrary{shapes,arrows,positioning}	
\usetikzlibrary {arrows.meta}	
\usetikzlibrary{decorations.markings,decorations.pathmorphing}

\newcommand{\rmidarrow}{\tikz \draw[thick,-{Computer Modern Rightarrow}] (-1pt,0) --(1pt,0);}	
\newcommand{\lmidarrow}{\tikz \draw[thick, -{Computer Modern Rightarrow}] (1pt,0) -- (-1pt,0);}

\usepackage[margin=1in]{geometry}

\newcommand{\mystar}{\operatornamewithlimits{\ast}}

\begin{document}

\newtheorem{thm}{Theorem}[section]
\newtheorem{lem}[thm]{Lemma}
\newtheorem{prop}[thm]{Proposition}
\newtheorem{cor}[thm]{Corollary}
\newtheorem{obs}[thm]{Observation}
\newtheorem{rem}[thm]{Remark}
\newtheorem{Def}[thm]{Definition}
\newtheorem{ex}[thm]{Example}

\title{Rigidity on Quantum Symmetry for a Certain  Class of Graph C*-algebras}
\author{Ujjal Karmakar \MakeLowercase{and} Arnab Mandal }
\maketitle

\begin{abstract}
	Quantum symmetry of graph $C^{*}$-algebras has been studied, under the consideration of different formulations, in the past few years. It is already known that the compact quantum group  $(\underbrace{C(S^{1})*C(S^{1})*\cdots *C(S^{1})}_{|E(\Gamma)|-times}, \Delta) $ [in short, $\left( \mystar\limits_{|E(\Gamma)|} C(S^1),  \Delta\right)$] always acts on a graph $C^*$-algebra for a finite, connected, directed graph  $\Gamma$ in the category introduced by Joardar and Mandal, where $|E(\Gamma)|:=$ number of edges in $\Gamma$. In this article, we show that for a certain class of graphs including Toeplitz algebra, quantum odd sphere, matrix algebra etc.  the  quantum symmetry of their associated graph $C^*$-algebras remains $\left( \mystar\limits_{|E(\Gamma)|} C(S^1),  \Delta\right) $ in the category as mentioned before. More precisely, if a finite, connected, directed graph $\Gamma$ satisfies the following graph theoretic properties : (i) there does not exist any cycle of length $\geq$ 2 (ii) there exists a path of length $(|V(\Gamma)|-1)$ which consists all the vertices, where $|V(\Gamma)|:=$ number of vertices in $\Gamma$ (iii) given any two vertices (may not be distinct) there exists at most one edge joining them, then the universal object coincides with $\left( \mystar\limits_{|E(\Gamma)|} C(S^1),  \Delta\right) $. Furthermore, we have pointed out a few counter examples whenever the above assumptions are violated. 	     
\end{abstract}

\maketitle


\section{Introduction}
In 1980, Cuntz and Krieger studied a family of $C^{*}$-algebras called Cuntz-Krieger algebras which are a rich supplier of examples for operator algebraists (refer to \cite{CK}). It can be treated as a generalized version of the Cuntz algebra which was discovered by Cuntz in \cite{Cuntz} considered as the first explicit example of a separable, simple, infinite $C^{*}$-algebra. A Graph $C^*$-algebra, a generalization of a Cuntz-Krieger algebra, is a universal $C^*$-algebra generated by some orthogonal projections and partial isometries coming from a given directed graph. We can capture a large class of examples of $C^*$-algebra as a graph $C^*$-algebra including matrix algebras, Toeplitz algebra, quantum spheres (odd and even), quantum projective space  etc. More interestingly, some operator algebraic properties of $C^*$-algebras can be recovered from the underlying graph and vice-versa. For instance, a graph $C^*$-algebra is unital (AF algebra) iff the underlying graph has only finitely many vertices (contains no cycle) (consult \cite{Kumjian}). It should be mentioned that  a graph $C^*$-algebra can be considered as a Cuntz-Krieger algebra if the underlying finite graph is highly connected (i.e. graph without sink).\\

On the other hand, groups founded in mathematics around the 19th century, as a collection of symmetries of an object. But the compact quantum group (in short, CQG) appeared in mathematics almost a hundred years after the appearance of the group by S.L. Woronowicz in \cite{Wor}. Also, some initial examples were constructed by him  using analytic construction. In noncommutative geometry, Mathematicians always wanted to find a right notion of symmetry for the noncommutative spaces. In 1995, Alain Connes raised the question to find an actual notion of quantum automorphism group for a  noncommutative space (see \cite{connes}). The core idea was to make a generalization of classical group symmetry to a `noncommutative version of symmetry'.  In 1998, Shuzhou Wang formulated the notion of a quantum automorphism group for a finite space $X_{n}$ containing $n$ points in \cite{Wang}. This problem was investigated in categorical language by him. It is known that the automorphism group of a space $X$, denoted by $Aut(X)$, can be realized as the universal object of the category whose objects are the faithful group actions $(G, \alpha)$ on that space $X$ and a morphism between two group actions is basically a group homomorphism between the underlying groups that respects the actions. Now in the quantum scenario, one just needs to replace the objects of the category by compact quantum groups and morphisms by compact quantum group morphisms. Though in the classical case, the universal object always does exist, in a quantum setting it may not be so.  So the first challenge is to show the existence of a universal object and the next challenge is to compute the exact quantum symmetry with respect to the respective category. For a $n$-set $X_n$, the function algebra $C(X_{n})$ is  isomorphic to $\mathbb{C}^n$ as a $C^*$-algebra. Wang has shown that though its classical automorphism group is $S_n$, the underlying $C^*$-algebra of the quantum automorphism group is non-commutative and infinite dimensional which is remarkably larger than $S_n$ for $n>3$ (see \cite{Wang}). Moreover, he also classified the quantum automorphism groups for any finite dimensional $C^*$-algebras but the problem was that the universal object fails to exist for other finite dimensional cases except $\mathbb{C}^n$ in the larger category. This issue was resolved by considering a subcategory of that large category via imposing a `volume' preserving condition (preserving a faithful state) on a suitable subspace and showed under this restricted set-up universal object exists on that subcategory. Later, T. Banica, J. Bichon extended the quantum symmetry structure on a finite graph  (see \cite{Ban}, \cite{Bichon}) and the notion of quantum isometry group (infinite dimensional set-up) was formulated by Goswami in \cite{Laplace}. A few years later, adopting the key ideas from their works, T. Banica and A. Skalski proposed the notion of orthogonal filtration on a $C^*$-algebra, equipped with a faithful state, and its quantum symmetry (refer to \cite{Ortho}). \\

The interesting fact is that though the function algebra over a finite, directed graph is finite dimensional, its associated graph $C^*$-algebra may indeed be infinite dimensional. So, it is natural to ask about the quantum symmetry of a graph $C^*$-algebra. S. Schmidt and M. Weber started the programme on it in an algebraic framework in \cite{Web}, whereas S. Joardar and A. Mandal have studied the quantum symmetry of a graph $C^*$-algebra in a more analytic framework. In \cite{Mandal}, the authors showed the existence of the universal object in their category and remarkably the quantum symmetry group of the graph $C^*$-algebra is strictly larger than the quantum automorphism group of the underlying graph in the sense of Banica. Moreover, they have shown that the CQG  $\left( \mystar\limits_{|E(\Gamma)|} C(S^1),  \Delta\right)$ always acts faithfully on any graph $C^*$-algebra for a finite, connected, directed graph. But it is natural to ask: when will it be the `largest' CQG acting faithfully on a graph $C^*$-algebra? It was already shown that the quantum symmetry is exactly $\left( \mystar\limits_{|E(\Gamma)|} C(S^1),  \Delta\right)$ for a simple directed path by Theorem 5.1 of \cite{Mandal}. Moreover,  Theorem 5.4 and  Proposition 5.9 of \cite{Mandal} tell us that the quantum symmetries are doubling of $C(S^{1})*C(S^{1})$ and $U_{n}^{+}$ for a complete graph with two vertices and Cuntz algebra (with n loops) respectively under the consideration of their category. Therefore, clearly one can't always expect the trivial quantum symmetry $\left( \mystar\limits_{|E(\Gamma)|} C(S^1),  \Delta\right)$ in general. In this article,  we are interested to find a certain class of graphs where the quantum symmetry coincides with  $\left( \mystar\limits_{|E(\Gamma)|} C(S^1),  \Delta\right)$ associated to their graph $C^*$-algebras. The stimulating fact is that this class covers several well known examples of $C^*$-algebras like $ C(S^1) $, Toeplitz algebra, odd quantum spheres, $ SU_{q}(2) $, even dimensional quantum balls etc. \\

Now, we briefly sketch the presentation of this article as follows: In the 2nd section, some basic facts are recalled about a directed graph, graph $C^*$-algebras, compact quantum groups and their action to a $C^*$-algebra, quantum automorphism groups. Moreover, we recall the quantum symmetry of a graph $C^*$-algebra with some definitions and proven results from \cite{Mandal}. In the 3rd section, we describe the class of graphs for which the associated graph $C^*$-algebra has $\left( \mystar\limits_{|E(\Gamma)|} C(S^1),  \Delta\right)$ as the universal object in the same category introduced in \cite{Mandal} regarding our context. Also, those graphs are characterized in terms of the adjacency matrix. In the 4th section, we prove the main theorem (Theorem \ref{T1}) by breaking this into two cases namely Proposition \ref{P1} and Proposition \ref{P2} with the help of some lemmas. In the last section, we provide some counter examples which tell us that if we deviate slightly from our desired class of graphs, the result may not hold.

\section{Preliminaries}

\subsection{Notations and conventions}
For a set $X$,  $|X|$ will denote the cardinality of $X$ and $ id_{X} $ will denote the identity function on $X$. For a $C^*$-algebra $ \mathcal{B} $, $\mathcal{B}^*$  is the set of all linear bounded functionals on $ \mathcal{B} $. For a set $X$, span($X$) will denote the linear space spanned by the elements of $X$. The tensor product `$\otimes$' denotes the spatial or minimal tensor product between two $C^*$-algebras.\\
For us, all the $C^*$-algebras are unital.

\subsection{Graph $C^{*}$-algebras}
In this subsection, we will recall some basic facts about graph $C^{*}$-algebra from \cite{BPRS, BHRS, Laca,  KPRR, Pask, Raeburn}.\\
A directed graph $\Gamma=\{ V(\Gamma), E(\Gamma),s,r \}$ consists of countable sets $ V(\Gamma) $ of vertices and $ E(\Gamma)$ of edges together with the maps $s,r: E(\Gamma) \to V(\Gamma) $ describing the source and range of the edges. We say $v \in V(\Gamma)$ is \textbf{adjacent to} $w \in V(\Gamma)$ (denoted by $v \to w $)  if there exists an edge $e \in E(\Gamma)$ such that  $ v=s(e)$ and $ w=r(e) $. A graph is said to be \textbf{finite} if both $|V(\Gamma)|$ and $|E(\Gamma)|$ are finite. Throughout this article, a directed graph is said to be \textbf{connected} if for every vertex $v \in V(\Gamma) $ at least one of $s^{-1}(v)$ or $r^{-1}(v)$ is nonempty. A \textbf{path} $\alpha$ of length $n$ in a directed graph $\Gamma$ is a sequence $\alpha=e_{1}e_{2} \cdots e_{n} $ of edges in $\Gamma$ such that $ r(e_{i}) = s(e_{i+1}) $ for $1 \leq i \leq (n-1) $. $s(\alpha):=s(e_{1})$ and $r(\alpha):=r(e_{n})$. A path of length $n$ in a directed graph $\Gamma$ is said to be a \textbf{cycle} of length $n$ if $s(\alpha)=r(\alpha)$ and $ s(e_i) \neq s(e_j)$ for $ i \neq j $. A \textbf{loop} is a cycle of length 1.\\
Let $\Gamma=\{V(\Gamma),E(\Gamma),s,r \}$ be a finite, directed graph with $|V(\Gamma)|=n$. The adjacency matrix of $\Gamma$ with respect to the ordering of the vertices $ (v_{1},v_{2},..., v_{n}) $ is a matrix $ (a_{ij})_{i,j= 1,2,...,n} $ with 
$ a_{ij} =
\begin{cases}
n(v_{i},v_{j}) & if ~~ v_{i} \to v_{j} \\
0 &  ~~ otherwise 
\end{cases} $ where $ n(v_{i},v_{j})$ denotes the number of edges joining $v_{i}$ to $v_{j}$.\\ 

\noindent In this article, we will define graph $C^*$-algebra only for a finite, directed graph.  
\begin{Def}
	Given a finite, directed graph $\Gamma$, the graph $C^{*}$-algebra $C^{*}(\Gamma)$ is a universal $C^{*}$-algebra generated by partial isometries $ \{S_{e}: e \in E(\Gamma) \} $ and orthogonal projections $ \{p_{v}: v \in V(\Gamma) \}$ such that \vspace{0.1cm}
	\begin{itemize}
		\item[(i)] $S_{e}^{*}S_{e}=p_{r(e)}$ for all $ e \in E(\Gamma) $. \vspace{0.1cm}
		\item[(ii)] $p_{v}=\sum\limits_{\{f:s(f)=v\}}S_{f}S_{f}^{*}$ for all $ v \in V(\Gamma) $ such that $s^{-1}(v) \neq \emptyset $.\\
	\end{itemize}
\end{Def} 

For any graph $C^{*}$-algebra, we have the following interesting results.
\begin{prop}
	For a finite, directed graph $\Gamma=\{V(\Gamma),E(\Gamma),s,r \}$
	\begin{itemize}
		\item[(i)] $S_{e}^{*}S_{f}=0$  for all $e \neq f $.\vspace{0.1cm}
		\item[(ii)] $\sum\limits_{v \in V(\Gamma)}p_{v}=1$.\vspace{0.1cm}
		\item[(iii)] $ S_{e}S_{f}\neq 0 \Leftrightarrow r(e)=s(f)$ i.e. $ef \text{ is a path} $.     \\
		Moreover, $S_{\gamma}:=S_{e_{1}}S_{e_{2}}...S_{e_{k}} \neq 0 \Leftrightarrow r(e_{i})=s(e_{i+1}) $ for $i=1,2,...,(k-1)$ i.e. $ \gamma=e_{1}e_{2}...e_{k}$ is a path.\vspace{0.1cm}
		\item[(iv)] $ S_{e}S_{f}^{*}\neq 0 \Leftrightarrow r(e)=r(f). $\vspace{0.1cm}
		\item[(v)] $ span\{S_{\gamma}S_{\mu}^{*} : \gamma, \mu \in E^{< \infty}(\Gamma) \text{ with } r(\gamma)=r(\mu)\} $ is dense in $C^{*}(\Gamma)$ where $E^{< \infty}(\Gamma)$ denotes the set of all finite length paths.\\
	\end{itemize} 
\end{prop}


\subsection{Compact quantum groups and quantum automorphism groups } \label{sec2.3}
In this subsection, we will recall some important facts related to compact quantum groups and their actions on a given $C^{*}$-algebra. We refer the readers to \cite{Van, Wang, Wor, Tim, Nesh} for more details.
\begin{Def}
	A compact quantum group(CQG) is a pair $(\mathcal{Q}, \Delta ) $, where $\mathcal{Q}$ is a unital $C^{*}$-algebra and $ \Delta : \mathcal{Q} \to \mathcal{Q} \otimes \mathcal{Q} $  is a unital $C^{*}$-homomorphism such that \vspace{0.1cm}
	\begin{itemize}
		\item[(i)] $(id_{\mathcal{Q}} \otimes \Delta)\Delta = (\Delta \otimes id_{\mathcal{Q}})\Delta $. \vspace{0.1cm}
		\item[(ii)] $ span\{ \Delta(\mathcal{Q})(1 \otimes \mathcal{Q} )\}$ and $ span\{\Delta(\mathcal{Q})(\mathcal{Q} \otimes 1)\} $ are dense in $(\mathcal{Q}\otimes \mathcal{Q})$. \\
	\end{itemize}
	
	Given two compact quantum groups $(\mathcal{Q}_{1},\Delta_{1})$ and $(\mathcal{Q}_{2},\Delta_{2})$, a compact quantum group morphism (CQG morphism) between $\mathcal{Q}_{1}$ and $\mathcal{Q}_{2}$ is a $C^{*}$-homomorphism $ \phi: \mathcal{Q}_{1} \to \mathcal{Q}_{2} $ such that $ (\phi \otimes \phi)\Delta_{1}=\Delta_{2}\phi $ .
\end{Def}

For any CQG $\mathcal{Q}$, there exists a canonical Hopf $*$-algebra $\mathcal{Q}_{0} \subseteq  \mathcal{Q} $ which is dense in $\mathcal{Q}$. Moreover, one can define an antipode $\kappa$ and a counit $\epsilon$ on the dense  Hopf $*$-algebra $\mathcal{Q}_{0}$.\\

\noindent Examples:
\begin{enumerate}
	\item Let $C(U_{n}^{+})$ be the universal $C^{*}$-algebra generated by $\{q_{ij}: i,j \in \{1,2,...,n\}\}$ such that $ U:=(q_{ij})_{n\times n} $ and $U^{t}$ both are unitary. Now, define $ \Delta : C(U_{n}^{+}) \to C(U_{n}^{+}) \otimes C(U_{n}^{+}) $ by $\Delta(q_{ij})=\sum_{k=1}^{n} q_{ik}\otimes q_{kj} $. Then $(U_{n}^{+},\Delta)$ denotes the CQG whose underlying $C^{*}$-algebra is $C(U_{n}^{+})$. (see \cite{Wangfree})\\
	\item For  $F \in \mathbb{GL}_{n}\mathbb{(C)}$, $ A_{U^{t}}(F)$ be the universal $C^{*}$-algebra generated by  $\{q_{ij}: i,j \in \{1,2,...,n\}\}$ such that \vspace{0.1cm}
	\begin{itemize}
		\item $U^{t}$ is unitary. \vspace{0.1cm}
		\item $UF^{-1}U^{*}F=F^{-1}U^{*}FU=Id_{n \times n}$. \vspace{0.1cm} 
	\end{itemize}
	Again coproduct is given on generators $\{q_{ij}\}_{i,j=1,2,...,n}$ by $\Delta(q_{ij})=\sum_{k=1}^{n} q_{ik}\otimes q_{kj} $. One can show that $(A_{U^{t}}(F),\Delta)$ is a CQG. Observe that $(A_{U^{t}}(Id_{n \times n}),\Delta)=(U_{n}^{+},\Delta)$.  (consult \cite{Van} for details)\\
	
	
	\item The commutative $C^{*}$-algebra $C(S^{1})$ can be thought of as a universal $C^{*}$-algebra generated by a unitary element $z$ i.e. 
	$C(S^{1})=\{z~|~zz^{*}=z^{*}z=1\}$. Now, define a coproduct on generator $z$ by $\Delta(z)=z \otimes z$. It is easy to check from the definition that $ (C(S^{1}),\Delta)$ forms a commutative CQG.\\[0.2cm]   
	Now, the free product of $m (\geq 2)$ copies of $C(S^{1})$ (denoted by $\left( \mystar\limits_{m} C(S^1),  \Delta\right)$ can be represented by $m$ unitary elements $ \{z_{i}\}_{i=1}^{m} $ i.e. $\{z_{1},z_{2},...,z_{m}~|~ z_{i}z_{i}^{*}=z_{i}^{*}z_{i}=1 ~\forall~ i \in \{1,2,...,m\} \}$. Define a coproduct $\Delta $ on $\{z_{i}\}_{i=1,...,m}$ by $\Delta(z_{i})=(z_{i} \otimes z_{i}) $ and 
	$\left( \mystar\limits_{m} C(S^1),  \Delta\right)$ is a CQG. It has the following universal property:\\
	If $(\mathcal{Q},\Delta)$ is a CQG generated by $m$ unitaries $\{u_{i}\}_{i=1,2,...,m}$ with $\Delta(u_{i})=u_{i} \otimes u_{i}$ for all $i \in \{1,2,...,m\}$, then there exist a surjective CQG morphism from $\left( \mystar\limits_{m} C(S^1),  \Delta\right)$ onto $(\mathcal{Q},\Delta)$ which sends $z_{i} \mapsto u_{i}$ for all $i \in \{1,2,...,m\}$. 
	(see \cite{Wangfree})\\
	
	\item The unitary easy CQG $H_{n}^{\infty +}$ is defined to be the universal $C^*$-algebra generated by $\{q_{ij}: i,j \in \{1,2,...,n\}\}$ such that \vspace{0.1cm}
	\begin{itemize}
		\item the matrices $(q_{ij})_{n \times n}$ and $(q_{ij}^{*})_{n \times n}$ are unitary. \vspace{0.1cm}
		\item $q_{ij}$'s are normal partial isometries for all $i,j$. \vspace{0.1cm}
	\end{itemize}
	The coproduct $\Delta $ on generators is again given by $\Delta(q_{ij})=\sum_{k=1}^{n} q_{ik}\otimes q_{kj}$. ( see \cite{easy} for details on the unitary easy quantum group ) \\
	
	\item Let $(Q, \Delta )$ be a compact quantum group with an automorphism $\varphi$ such that $\varphi^{2}=id_{Q}$. The doubling of $Q$ is a CQG (denoted as $ (\mathcal{D}_{\varphi}(Q),\Delta_{\varphi})$)  whose underlying $C^*$-algebra is $Q \oplus Q$ and coproduct $ \Delta_{\varphi} $ is defined by 
	$$ \Delta_{\varphi}\circ \zeta =(\zeta \otimes \zeta + \eta \otimes [\eta \circ \varphi])\circ \Delta , $$ 
	$$ \Delta_{\varphi}\circ \eta =(\zeta \otimes \eta + \eta \otimes [\zeta\circ \varphi])\circ \Delta , $$ 
	where $ \zeta,\eta : Q \to Q \oplus Q $ such that             
	$\zeta(a)=(a,0)$ and $\eta(b)=(0,b)$ (for the construction see \cite{Soltan}).\\
\end{enumerate}

For the following definitions and discussions of this subsection, readers are referred to \cite{Wang} and \cite{Bichon}.
\begin{Def}
	A CQG $(\mathcal{Q},\Delta)$ is said to be act faithfully on a unital $C^{*}$-algebra $\mathcal{C}$ if there exists a unital $C^{*}$-homomorphism $\alpha:\mathcal{C} \to \mathcal{C} \otimes \mathcal{Q} $ such that \vspace{0.1cm}
	\begin{itemize}
		\item[(i)]$(\alpha \otimes id_{\mathcal{Q}})\alpha = (id_{\mathcal{C}} \otimes \Delta)\alpha.$ \vspace{0.1cm}
		\item[(ii)] $span\{\alpha(\mathcal{C})(1\otimes \mathcal{Q})\}$ is dense in $\mathcal{C} \otimes \mathcal{Q}.$ \vspace{0.1cm}
		\item[(iii)] The $*$-algebra generated by the set $\{(\theta\otimes id)\alpha(\mathcal{C}) : \theta \in \mathcal{C}^{*}\}$ is norm-dense in $\mathcal{Q}$.\vspace{0.1cm}
	\end{itemize}
\end{Def}
$((\mathcal{Q},\Delta),\alpha)$ is also called \textbf{quantum transformation group of $\mathcal{C} $}.\\

Given a unital $C^{*}$-algebra $ \mathcal{C} $, one can introduce a \textbf{category} $ \mathfrak{C} $ whose objects are quantum transformation groups of $\mathcal{C}$ and morphism from $((\mathcal{Q}_{1},\Delta_{1}),\alpha_{1})$ to $((\mathcal{Q}_{2},\Delta_{2}),\alpha_{2})$  be a CQG morphism $\phi :(\mathcal{Q}_{1},\Delta_{1}) \to (\mathcal{Q}_{2},\Delta_{2})$ such that $ (id_{\mathcal{C}} \otimes \phi)\alpha_{1}=\alpha_{2} $. It is called \textbf{category of quantum transformation group of $\mathcal{C}$}.\\
The \textbf{universal object of the category $ \mathfrak{C} $} be a quantum transformation group of $\mathcal{C}$, denoted by $((\widehat{\mathcal{Q}},\widehat{\Delta}),\widehat{\alpha})$, satisfying the following universal property :\\
For any object $((\mathcal{B}, \Delta_{\mathcal{B}}),\beta)$ from the category of quantum transformation group of $\mathcal{C}$, there is a surjective CQG morphism $\widehat{\phi}: (\widehat{\mathcal{Q}}, \widehat{\Delta}) \to (\mathcal{B}, \Delta_{\mathcal{B}}) $ such that $(id_{\mathcal{C}} \otimes \widehat{\phi})\widehat{\alpha}=\beta $. 

\begin{Def}
	Given a unital $C^{*}$-algebra $\mathcal{C}$, the quantum automorphism group of $ \mathcal{C} $ is the underlying CQG of the universal object of the category of quantum transformation group of $\mathcal{C}$ if the universal object exists.
\end{Def}

\begin{rem}
	In the above category, the universal object might fail to exist in general. For the existence of a universal object, one usually restricts the category to a sub-category in the following manner:\\
	Fix a linear functional $ \tau: \mathcal{C} \to \mathbb{C} $. Now, define a subcategory $\mathfrak{C}_{\tau}$ whose objects are those quantum transformation group of $\mathcal{C}$, $((\mathcal{Q},\Delta),\alpha)$ for which $(\tau \otimes id)\alpha(.)=\tau(.).1 $ on a suitable subspace of $\mathcal{C}$ and morphisms are taken as the above.  
\end{rem}

\noindent Examples:
\begin{enumerate}
	\item For $n$ points space $X_{n}$, the universal object in the category of quantum transformation group of $C(X_{n})$ exists and the underlying $C^{*}$-algebra of quantum automorphism group of $C(X_{n})$ is the universal $C^{*}$-algebra generated by $\{u_{ij}\}_{i,j=1,2,...,n}$ such that the following relations are satisfied: \vspace{0.1cm}
	\begin{itemize}
		\item  $u_{ij}^{2}=u_{ij}=u_{ij}^{*} $ for all $i,j \in \{1,2,...,n\}$, \vspace{0.1cm}
		\item  $\sum_{k=1}^{n} u_{ik}= \sum_{k=1}^{n} u_{kj}=1 $ for all $i,j \in \{1,2,...,n\}$. \vspace{0.1cm}
	\end{itemize} 
	Moreover, the coproduct on generators is given by $\Delta(u_{ij})= \sum_{k=1}^{n}u_{ik} \otimes u_{kj}$. Then the quantum automorphism group of $C(X_{n})$ is the quantum permutation group, $S_{n}^{+}$  (see \cite{Wang, BBC} for more details).\\
	
	\item For the $C^{*}$-algebra $ M_{n}(\mathbb{C}) $, the universal object in the category of quantum transformation group of $ M_{n}(\mathbb{C}) $ (for $n \geq 2 $) does not exist. But if we fix a linear functional $\tau'$ on $ M_{n}(\mathbb{C}) $ which is defined by $\tau'(A)=Tr(A)$ and assume that any object of the category also preserves $\tau'$ i.e. $(\tau' \otimes id)\alpha(.)=\tau'(.).1 $ on $M_{n}(\mathbb{C})$, then the universal object would exist in $\mathfrak{C}_{\tau'}$ (see \cite{Wang} for more details).
\end{enumerate}

\subsection{Quantum symmetry of a graph $C^{*}$-algebra}
Let $\Gamma=\{V(\Gamma),E(\Gamma),r,s\}$ be a finite, connected graph. Since $\Gamma$ is connected, it is enough to define an action on the partial isometries corresponding to edges. 
\begin{Def}(Definition 3.4 of \cite{Mandal})
	Given a connected graph $\Gamma$, a faithful action $\alpha$ of a CQG $\mathcal{Q}$ on a $C^{*}$-algebra $C^{*}(\Gamma)$ is said to be linear if $ \alpha(S_{e})=\sum_{f \in E(\Gamma)} S_{f} \otimes q_{fe}$, where $q_{ef}\in \mathcal{Q}$ for each $e,f \in E(\Gamma)$.
\end{Def}

\noindent Let 
\begin{itemize}
	\item[1)] $\mathcal{I}=\{u\in V(\Gamma) : u \text{ is not a source of any edge of } \Gamma\}$ \vspace{0.1cm}
	\item[2)] $ E'= \{(e,f) \in E(\Gamma) \times E(\Gamma) : S_{e}S_{f}^{*} \neq 0 \}=\{(e,f) \in E(\Gamma) \times E(\Gamma) : r(e)=r(f) \}$ \vspace{0.1cm}
\end{itemize}
$\bullet$ It can be shown that $\{ p_{u}, S_{e}S_{f}^{*} : u \in \mathcal{I}, (e,f) \in E' \}$ is a linearly independent set. (Lemma 3.2 of \cite{Mandal})\\[0.2cm]
Now, define $\mathcal{V}_{2,+}= span\{ p_{u}, S_{e}S_{f}^{*} : u \in \mathcal{I}, (e,f) \in E' \}$ and a linear functional $\tau: \mathcal{V}_{2,+} \to \mathbb{C}$ by $\tau(S_{e}S_{f}^{*})=\delta_{ef}$, $\tau(p_{u})=1 $ for all $ (e,f) \in E' $ and $ u \in \mathcal{I}$.
(see subsection 3.1 of \cite{Mandal})\\[0.1cm]
$\bullet$ One can check that $\alpha(\mathcal{V}_{2,+}) \subseteq \mathcal{V}_{2,+} \otimes \mathcal{Q}$. (Lemma 3.6 of \cite{Mandal})\\
Therefore the equation $(\tau \otimes id)\alpha(.)=\tau(.).1 $ on $\mathcal{V}_{2,+}$ makes sense.\\

\begin{Def}(Definition 3.7 of \cite{Mandal})
	For a finite, connected graph $\Gamma $, define a category $\mathfrak{C}_{\tau}^{Lin} $ whose objects are $((\mathcal{Q},\Delta),\alpha) $, quantum transformation group of $ C^{*}(\Gamma)$ such that $(\tau \otimes id)\alpha(.)=\tau(.).1 $ on $\mathcal{V}_{2,+}$. Morphism from $((\mathcal{Q}_{1},\Delta_{1}),\alpha_{1})$ to $((\mathcal{Q}_{2},\Delta_{2}),\alpha_{2})$ is given by a CQG morphism $\Phi:\mathcal{Q}_{1} \to \mathcal{Q}_{2} $ such that $ (id_{C^{*}(\Gamma)} \otimes \phi)\alpha_{1}=\alpha_{2} $.
\end{Def}

$F^{\Gamma}$ is a $ (|E(\Gamma)| \times |E(\Gamma)| ) $ matrix such that $(F^{\Gamma})_{ef}= \tau(S_{e}^{*}S_{f})$. It can be shown that $F^{\Gamma}$ is an invertible diagonal matrix. Therefore, $A_{U^{t}}(F^{\Gamma})$ is a CQG. We refer the readers to Proposition 3.8 and Theorem 3.9 of \cite{Mandal} for the proof of the following theorem.          
\begin{thm}
	For a finite connected graph $\Gamma$,\vspace{0.1cm}
	\begin{enumerate}
		\item there is a surjective $C^{*}$-homomorphism from $A_{U^{t}}(F^{\Gamma})$ to any object in category $\mathfrak{C}_{\tau}^{Lin}$. \\
		\item the category $\mathfrak{C}_{\tau}^{Lin} $ admits a universal object.\\
	\end{enumerate}
	We denote the universal object by $ Q_{\tau}^{Lin} $ in category $\mathfrak{C}_{\tau}^{Lin} $.\\
\end{thm}

Recall from Subsection \ref{sec2.3} that $\left( \mystar\limits_{m} C(S^1),  \Delta\right)$ has a CQG structure. The following proposition tells us that this CQG always acts on a graph $C^{*}$-algebra $C^{*}(\Gamma)$ for a finite, connected graph $\Gamma$. For more details consult Proposition 3.11 and Corollary 3.12 from \cite{Mandal}.
\begin{prop}
	For a finite, connected, directed graph  $\Gamma=\{V(\Gamma),E(\Gamma),r,s\}$, \\ $\left(\mystar\limits_{|E(\Gamma)|} C(S^1),  \Delta\right)$ always belongs to category $\mathfrak{C}_{\tau}^{Lin} $, where $ \alpha(S_{e})=S_{e} \otimes q_{e} $ for all $e \in E(\Gamma)$ and $\{q_{e}\}_{e \in E(\Gamma)}$ are unitaries generating $\left(\mystar\limits_{|E(\Gamma)|} C(S^1),  \Delta\right)$.\\
	Moreover, the underlying $C^{*}$-algebra of $Q_{\tau}^{Lin}$ is noncomutative for a graph $\Gamma $ with $|E(\Gamma)| \geq 2 $.\\
\end{prop}

Next, we will see an example of graph $C^{*}$-algebra for which universal object in category $\mathfrak{C}_{\tau}^{Lin} $ will be exactly $\left(\mystar\limits_{|E(\Gamma)|} C(S^1),  \Delta\right)$. Based on this example, our interest is to find a certain class of graphs $\mathcal{G}$, so that the universal object of $\mathfrak{C}_{\tau}^{Lin}$ remains $\left(\mystar\limits_{|E(\Gamma)|} C(S^1),  \Delta\right)$ for every graph from $\mathcal{G}$.\\

\noindent Example:
\begin{enumerate}
	\item Consider a graph $P_{m+1}$ with $ (m+1)$ vertices $\{1,2,...,m+1\}$ and $m$ consecutive edges $\{e_{12},e_{23},...,e_{m(m+1)}\}$ joining them, i.e. $s(e_{ij})=i$ and $r(e_{ij})=j$.  
	Using the relations of graph $C^{*}$-algebra it can be shown that $C^{*}(P_{m+1})$ is $C^{*}$-isomorphic to $M_{(m+1)}(\mathbb{C})$. In Theorem 4.1 in \cite{Wang}, Wang has described the universal object of category  $\mathfrak{C}_{\tau'}$ for $M_{n}(\mathbb{C})$. But viewing $M_{(m+1)}(\mathbb{C})$ as a graph $C^{*}$-algebra $C^{*}(P_{m+1})$, in \cite{Mandal} it  has been shown the universal object is $\left(\mystar\limits_{m} C(S^1),  \Delta\right)$ in category $\mathfrak{C}_{\tau}^{Lin} $. That is,  
\end{enumerate}

\begin{thm}
	For graph $P_{m+1}$, $Q_{\tau}^{Lin}$ is isomorphic to $\left(\mystar\limits_{m} C(S^1),  \Delta \right)$. \ (Theorem 4.2 of \cite{Mandal})
\end{thm}


\section{Discussion on a special class of graph $C^{*}$-algebras} 
In this section, we will provide a special class of graphs $\mathcal{G}$ such that $Q_{\tau}^{Lin} $  corresponding to a graph $C^{*}$-algebra with respect to a graph $\Gamma \in \mathcal{G}$ is isomorphic to
$\left(\mystar\limits_{|E(\Gamma)|} C(S^1),  \Delta \right)$.\\

We will say that a graph $\Gamma=\{V(\Gamma),E(\Gamma),r,s\}$ satisfies the property \textbf{(R)} if it satisfies the following conditions:\vspace{0.1cm}
\begin{itemize}
	\item[\textbf{(R1)}] There does not exist any cycle of length $\geq$ 2. \vspace{0.2cm}
	\item[\textbf{(R2)}] There exists a path of length $(|V(\Gamma)|-1)$ which consists all the vertices.\vspace{0.2cm}
	\item[\textbf{(R3)}] Given any two vertices (may not be distinct) there exists at most one edge joining them. \vspace{0.2cm}
\end{itemize}
From now on, $\mathcal{G}$ is the set of all finite, directed graphs satisfying the property $\textbf{(R)}$.\\

\noindent Based on the properties \textbf{(R1), (R2)} and \textbf{(R3)}, in the next proposition, we will characterize the adjacency matrix of any graph $ \Gamma \in  \mathcal{G} $.
\begin{prop}
	\label{P}
	A graph $\Gamma=\{V(\Gamma),E(\Gamma),r,s\}$ satisfies the property \textbf{(R)} iff there exists an ordering of the vertices such that the adjacency matrix of $\Gamma$ with respect to that ordering is of the form $(a_{ij})_{n \times n}$ where $ a_{ij}\in \begin{cases}
	\{0\} ~ if ~ i>j \\ 
	\{1\} ~ if ~ j=i+1 \\
	\{0,1\} ~ otherwise
	\end{cases} $ \\
	i.e. adjacency matrix is an upper triangular 0-1 matrix with super diagonal entries all are 1.                                               
	
\end{prop}
\begin{proof}
	Assume that $|V(\Gamma)|=n$ and $\gamma = e_{1}e_{2}...e_{n-1}$ is a path of length $(n-1)$ which contains all the vertices where each $e_{i}$ is an edge (i.e. a path of length 1). Firstly, observe that $e_{i}$ can't be a loop for each $i=1,2,...,n-1$. If it is so for some $e_j$, then $ s(e_{j})=r(e_{j})$. Therefore, clearly the path $\gamma = e_{1}e_{2}...e_{n-1}$ consists at most $(n-1)$ vertices which is a contradiction. Now, we denote $s(\gamma)=s(e_{1})=1$, $ s(e_{i})=r(e_{i-1})=i$ for each $(n-1)\geq i>1$ and  $r(e_{n-1})= n $. With respect to this ordering $(1,2,...,n)$ of vertices, $a_{i(i+1)} \geq 1 $ for each $1\leq i\leq (n-1)$. In the presence of \textbf{(R3)}, we get  $a_{i(i+1)} = 1 $ for each $1\leq i\leq (n-1)$. Now given any two vertices $i,j$ with $i>j$, if there exists an edge from $i$ to $j$ then the graph has a cycle of length $\geq 2$ which is a contradiction to \textbf{(R1)}. Therefore, $a_{ij}=0$ with $i>j$. Lastly, by \textbf{(R3)} there is at most one path joining $i$ to $j$, i.e. $a_{ij}\in \{0,1\}$ for $j-i>1$. Moreover, there is at most one loop at each vertex i.e. $a_{ii}\in \{0,1\}$ for all $i \in \{1,2,...,n\}$. \\
	
	Conversely, assume that the adjacency matrix of $ \Gamma $ is of the form $(a_{ij})$ where $ a_{ij}\in \begin{cases}
	\{0\} ~ if ~ i>j \\ 
	\{1\} ~ if ~ j=i+1 \\
	\{0,1\} ~ otherwise
	\end{cases} $ \\ 
	with respect to ordering $(v_{1},v_{2},...,v_{n})$.
	Since $(a_{ij})$ is an upper triangular matrix, the graph contains no cycle of length $\geq 2$.
	There is exactly one edge $e_{i}$ joining $v_{i}$ to $v_{i+1}$ for all $i=1,2,...,(n-1)$. This means there exists a path $e_{1}e_{2}...e_{n-1}$ of length $n-1$ which contains all the vertices. Since each $a_{ij} \in \{0,1\}$, there is at most one edge joining $v_{i}$ and $v_{j}$.                                     
	
\end{proof}

Now, we will identify some well known examples of graph $C^{*}$-algebras whose underlying graphs belong to $\mathcal{G}$ (Figs. \ref{Pn}-\ref{M4}).\\

\textbf{Examples:}
\begin{enumerate}
	\item Let $ P_{n} $ (Fig. \ref{Pn}) be a graph with  adjacency matrix $(a_{ij})_{n \times n}$ such that
	$ a_{ij}= \begin{cases}
	0 ~~ if ~~ i>j ~ \text{and} ~ (j-i)>1 \\ 
	1 ~~ if ~~ j=i+1 
	\end{cases} $ 
	then $P_{n} \in \mathcal{G}$ and $C^*(P_{n})$ is $C^*$-isomorphic $M_{n}(\mathbb{C})$ (see \cite{MS}).                                                               
	
	\begin{center}
		\begin{figure}[htb] \label{Pn}
			\begin{tikzpicture}
			\draw[fill=black] (0,0) circle (2pt) node[anchor=north]{$v_{1}$};
			\draw[fill=black] (2,0) circle (2pt) node[anchor=north]{$v_{2}$};
			\draw[fill=black] (4,0) circle (2pt) node[anchor=north]{$v_{3}$};
			\draw[fill=black] (6,0) circle (2pt) node[anchor=north]{$v_{n-1}$};
			\draw[fill=black] (8,0) circle (2pt) node[anchor=north]{$v_{n}$};
			
			\draw[black,thick] (0,0)-- node{\rmidarrow}  node[above]{$e_{12}$} (2,0);
			\draw[black,thick] (2,0)-- node{\rmidarrow} node[above]{$e_{23}$} (4,0);
			\draw[loosely dotted, black,thick] (4.5,0)-- (5.5,0);
			\draw[black,thick] (6,0)--node{\rmidarrow}  node[above]{$e_{(n-1)n}$} (8,0);
			\end{tikzpicture}
			\caption{$P_{n}$}
		\end{figure}
	\end{center} 
	
	\item
	Let $ T $ (Fig. \ref{T}) be a graph with 2 vertices whose adjacency matrix is given by $ \begin{bmatrix}
	1 & 1 \\
	0 & 0
	\end{bmatrix} $.
	Then   $T \in \mathcal{G}$ and $C^*(T)$ is $C^*$-isomorphic Toeplitz algebra  $\mathcal{T}$ (for details consult \cite{MS}). 
	
	\begin{center}
		\begin{figure}[htb] \label{T}
			\begin{tikzpicture}
			\draw[fill=black] (0,0) circle (2pt) node[anchor=south]{$v_{1}$};
			\draw[fill=black] (2,0) circle (2pt) node[anchor=south]{$v_{2}$};
			
			\begin{scope}[decoration={markings,
				mark=at position 0.25 with {\arrow[black,thick]{<}},
			}]
			\draw[black,postaction={decorate}, thick](0,0.5) circle (0.5)  node[above=0.6]{$e_{11}$};
			\end{scope}
			
			\begin{scope}[decoration={markings,
				mark=at position 0.50 with {\arrow[black,thick]{>}},
			}]
			\draw[black,postaction={decorate}, thick ](0,0)--node[above]{$e_{12}$}  (2,0);
			\end{scope}
			
			\end{tikzpicture}
			\caption{$T$}
		\end{figure}
	\end{center}                                                            
	
	\item Let $ L_{2n-1} $ (Fig. \ref{L9}) be a graph with  adjacency matrix $(a_{ij})_{n \times n}$ such that
	$ a_{ij}= \begin{cases}
	0 ~~ if ~~ i>j \\ 
	1 ~~ if ~~ i \leq j
	\end{cases} $ 
	then $L_{2n-1} \in \mathcal{G}$ and $C^*(L_{2n-1})$ is $C^*$-isomorphic to the underlying $C^*$-algebra of  odd quantum sphere $C(S_{q}^{2n-1})$ \ \ (for details consult Theorem 4.4 of \cite{HSproj}). \\
	
	\begin{center}
		\begin{figure}[htb] \label{L9}
			\begin{tikzpicture}
			\draw[fill=black] (0,0) circle (2pt) node[anchor=south]{$v_{1}$};
			\draw[fill=black] (2,0) circle (2pt) node[anchor=south]{$v_{2}$};
			\draw[fill=black] (4,0) circle (2pt) node[anchor=south]{$v_{3}$};
			\draw[fill=black] (6,0) circle (2pt) node[anchor=south]{$v_{4}$};
			\draw[fill=black] (8,0) circle (2pt) node[anchor=south]{$v_{5}$};
			
			\draw[black,thick](0,0.5) circle (0.5) node[above=0.25]{\rmidarrow}node[above=0.6]{$e_{11}$};
			\draw[black,thick](2,0.5) circle (0.5) node[above=0.25]{\rmidarrow}  node[above=0.6]{$e_{22}$};
			\draw[black,thick](4,0.5) circle (0.5) node[above=0.25]{\rmidarrow}  node[above=0.6]{$e_{33}$};
			\draw[black,thick](6,0.5) circle (0.5)  node[above=0.25]{\rmidarrow} node[above=0.6]{$e_{44}$};
			\draw[black,thick](8,0.5) circle (0.5)  node[above=0.25]{\rmidarrow} node[above=0.6]{$e_{55}$};
			
			\draw[black,thick] (0,0)-- node{\rmidarrow}  node[above]{$e_{12}$} (2,0);
			\draw[black,thick] (2,0)-- node{\rmidarrow} node[above]{$e_{23}$} (4,0);
			\draw[black,thick] (4,0)--node{\rmidarrow}  node[above]{$e_{34}$} (6,0);
			\draw[black,thick] (6,0)--node{\rmidarrow}  node[above]{$e_{45}$} (8,0);
			
			\draw[black,thick] (0,0) edge[bend left= -45] node{\rmidarrow}  node[above]{$e_{14}$} (4,0);
			\draw[black,thick] (0,0) edge[bend left= -55] node{\rmidarrow}  node[above]{$e_{15}$} (6,0);
			\draw[black,thick] (0,0) edge[bend left= -65] node{\rmidarrow}  node[above]{$e_{15}$} (8,0);
			\draw[black,thick] (2,0) edge[bend left= -45] node{\rmidarrow}  node[above]{$e_{24}$} (6,0);
			\draw[black,thick] (2,0) edge[bend left= -55] node{\rmidarrow}  node[above]{$e_{25}$} (8,0);
			\draw[black,thick] (4,0) edge[bend left= -45] node{\rmidarrow}  node[above]{$e_{35}$} (8,0);
			
			\end{tikzpicture}
			\caption{$L_{9}$}
		\end{figure}
	\end{center}
	
	\item  Let $ \overline{L}_{2n-1} $ (Fig. \ref{Lbar2n-1})be a graph with n vertices whose adjacency matrix is of the form $(a_{ij})_{n \times n}$ where $a_{ij}=  \begin{cases}
	1 ~if ~ i=j \text{ and } j=i+1 \\
	0 ~ otherwise
	\end{cases}$, i.e. a Jordan block with all the eigenvalues are 1. 
	Again  ${\overline{L}_{2n-1}} \in \mathcal{G} $ and $C^{*}(\overline{L}_{2n-1})$ is $C^*$-isomorphic to  $C^*$-algebra $C(S_{q}^{2n-1})$ for $q \in [0,1)$. Moreover, it is well known that $C^{*}(\overline{L}_{3})$ is $C^*$-isomorphic to $C(SU_{q}(2))$ \ \ (see appendix A from \cite{HSproj}).                                     
	
	\begin{center}
		\begin{figure}[htpb] \label{Lbar2n-1}
			\begin{tikzpicture}
			\draw[fill=black] (0,0) circle (2pt) node[anchor=south]{$v_{1}$};
			\draw[fill=black] (2,0) circle (2pt) node[anchor=south]{$v_{2}$};
			\draw[fill=black] (4,0) circle (2pt) node[anchor=south]{$v_{3}$};
			\draw[fill=black] (6,0) circle (2pt) node[anchor=south]{$v_{n}$};

			\draw[black,thick](0,0.5) circle (0.5) node[above=0.25]{\rmidarrow}  node[above=0.6]{$e_{11}$};
			\draw[black,thick](2,0.5) circle (0.5) node[above=0.25]{\rmidarrow}  node[above=0.6]{$e_{22}$};
			\draw[black,thick](4,0.5) circle (0.5) node[above=0.25]{\rmidarrow}  node[above=0.6]{$e_{33}$};
			\draw[black,thick](6,0.5) circle (0.5) node[above=0.25]{\rmidarrow}  node[above=0.6]{$e_{nn}$};

			\draw[black,thick] (0,0)-- node{\rmidarrow}  node[above]{$e_{12}$} (2,0);
			\draw[black,thick] (2,0)-- node{\rmidarrow}  node[above]{$e_{23}$} (4,0);
			\draw[black,thick] (4,0)-- (4.4,0);
			\draw[black,thick] (5.6,0)-- (6,0);
			\draw[loosely dotted, black,thick] (4.6,0)-- (5.4,0);
			\end{tikzpicture}
			\caption{ $ \overline{L}_{2n-1} $}
		\end{figure}
	\end{center}
	
	\item Let $ M_{n} $ (Fig. \ref{M4}) be a graph whose adjacency matrix is a $ (n+1) \times (n+1) $ such that  $ a_{ij}= \begin{cases}
	0 ~~ if ~~ i>j \\
	0 ~~if ~~ i=j=(n+1)\\ 
	1 ~~ otherwise
	\end{cases} $, then $M_{n} \in \mathcal{G}$ and $C^*(M_{n})$ is $C^*$-isomorphic to the underlying $C^*$-algebra of even dimensional quantum ball, $C(B_{q}^{2n})$ (for more details consult \cite{HSball}). 
	
	\begin{center}
		\begin{figure}[htpb] \label{M4}
			\begin{tikzpicture}
			\draw[fill=black] (0,0) circle (2pt) node[anchor=south]{$v_{1}$};
			\draw[fill=black] (2,0) circle (2pt) node[anchor=south]{$v_{2}$};
			\draw[fill=black] (4,0) circle (2pt) node[anchor=south]{$v_{3}$};
			\draw[fill=black] (6,0) circle (2pt) node[anchor=south]{$v_{4}$};
			\draw[fill=black] (8,0) circle (2pt) node[anchor=south]{$v_{5}$};
			\draw[black,thick](0,0.5) circle (0.5) node[above=0.25]{\rmidarrow}node[above=0.6]{$e_{11}$};
			\draw[black,thick](2,0.5) circle (0.5) node[above=0.25]{\rmidarrow}  node[above=0.6]{$e_{22}$};
			\draw[black,thick](4,0.5) circle (0.5) node[above=0.25]{\rmidarrow}  node[above=0.6]{$e_{33}$};
			\draw[black,thick](6,0.5) circle (0.5)  node[above=0.25]{\rmidarrow} node[above=0.6]{$e_{44}$};
			
			\draw[black,thick] (0,0)-- node{\rmidarrow}  node[above]{$e_{12}$} (2,0);
			\draw[black,thick] (2,0)-- node{\rmidarrow} node[above]{$e_{23}$} (4,0);
			\draw[black,thick] (4,0)--node{\rmidarrow}  node[above]{$e_{34}$} (6,0);
			\draw[black,thick] (6,0)--node{\rmidarrow}  node[above]{$e_{45}$} (8,0);
			
			\draw[black,thick] (0,0) edge[bend left= -45] node{\rmidarrow}  node[above]{$e_{14}$} (4,0);
			\draw[black,thick] (0,0) edge[bend left= -55] node{\rmidarrow}  node[above]{$e_{15}$} (6,0);
			\draw[black,thick] (0,0) edge[bend left= -65] node{\rmidarrow}  node[above]{$e_{15}$} (8,0);
			\draw[black,thick] (2,0) edge[bend left= -45] node{\rmidarrow}  node[above]{$e_{24}$} (6,0);
			\draw[black,thick] (2,0) edge[bend left= -55] node{\rmidarrow}  node[above]{$e_{25}$} (8,0);
			\draw[black,thick] (4,0) edge[bend left= -45] node{\rmidarrow}  node[above]{$e_{35}$} (8,0);
			
			\end{tikzpicture}
			\caption{$M_{4}$}
		\end{figure}
	\end{center}
	
\end{enumerate}

\section{ Main Result }
Let $\Gamma=\{ V(\Gamma),E(\Gamma),r,s \}$ be a finite, connected, directed graph. We will show that the quantum symmetry in category $\mathfrak{C}_{\tau}^{Lin}$ corresponding to a graph $C^{*}$-algebra with respect to a graph $\Gamma \in \mathcal{G}$ is isomorphic to
$\left(\mystar\limits_{|E(\Gamma)|} C(S^1),  \Delta\right)$.\\

Assume that $V(\Gamma)=\{1,2,...,n\}$ and let the adjacency matrix of $\Gamma=(a_{ij})_{n \times n}$ be of the form as in Proposition \ref{P} (i.e. $ \Gamma \in \mathcal{G}$). Moreover, assume that $ Q_{\tau}^{Lin}$ is generated by the elements of the matrix $(q_{ef})_{|E(\Gamma)| \times |E(\Gamma)|}$ where $e,f \in E(\Gamma)$.\\

\noindent We will prove our main result, namely Theorem \ref{T1} by showing Proposition \ref{P1} and Proposition \ref{P2}. To prove those propositions, we need the help of the following auxiliary lemmas.

\begin{lem}
	\label{l1}
	If  $\Gamma=\{ V(\Gamma),E(\Gamma),r,s \}$ is a finite, connected, directed graph satisfying the following conditions:\vspace{0.1cm}
	\begin{itemize}
		\item[(H 1.1)] there exists some $ h \in E(\Gamma)$ such that $ s^{-1}(r(h)) \neq \emptyset $,\vspace{0.1cm}
		\item[(H 1.2)] there exists $ e \in E(\Gamma)$ such that either $ r^{-1}(s(e))=\emptyset $ or $q_{gh}=0$ for all $g \in r^{-1}(s(e))\neq \emptyset $.  
	\end{itemize}
	\vspace{0.1cm}	Then $q_{ef}=0 $ for all $f$ such that $s(f)=r(h)$.\\
\end{lem}

\begin{proof}
	
	At $r(h)$, we can write
	\begin{align*}
		& p_{r(h)}=S_{h}^{*}S_{h}=\sum_{\{f:s(f)=r(h)\}} S_{f}S_{f}^{*} ~~~~ [ \text{ by (H 1.1) }] \\
		\Rightarrow  & ~~\alpha \left( S_{h}^{*}S_{h} \right)= \alpha \left( \sum_{\{f:s(f)=r(h)\}} S_{f}S_{f}^{*} \right)  \\
		\Rightarrow & \sum_{g\in E(\Gamma)} S_{g}^{*}S_{g}\otimes q_{gh}^{*}q_{gh}= \sum_{k,l\in E(\Gamma)}S_{k}S_{l}^{*} \otimes \left( \sum_{\{f:s(f)=r(h)\}} q_{kf}q_{lf}^{*}\right)
	\end{align*}
	
	Multiplying both sides of the above equation by $ (S_{e}^{*} \otimes 1)$ from left and $(S_{e} \otimes 1)$ from right, we get
	
	\begin{align*}
		& \sum_{g\in E(\Gamma)}S_{e}^{*}S_{g}^{*}S_{g}S_{e}\otimes q_{gh}^{*}q_{gh}= \sum_{k,l\in E(\Gamma)}S_{e}^{*}S_{k}S_{l}^{*}S_{e} \otimes \left( \sum_{\{f:s(f)=r(h)\}} q_{kf}q_{lf}^{*}\right) \\
		\Rightarrow  & ~~ 0= S_{e}^{*}S_{e}S_{e}^{*}S_{e} \otimes \left( \sum_{\{f:s(f)=r(h)\}}q_{ef}q_{ef}^{*} \right)~~~~[\text{ by (H 1.2) and Proposition 2.2(iii) } ]\\ 
		\Rightarrow & ~~ S_{e}^{*}S_{e} \otimes \left( \sum_{\{f:s(f)=r(h)\}}q_{ef}q_{ef}^{*} \right)=0  \\
		\Rightarrow & \sum_{\{f:s(f)=r(h)\}}q_{ef}q_{ef}^{*}=0\\
		\Rightarrow & ~~ q_{ef}q_{ef}^{*}=0 \text{ for all } f ~\text{such that}~ s(f)=r(h) \\
		\Rightarrow & ~~ q_{ef}=0 \text{ for all } f ~\text{such that}~ s(f)=r(h).
	\end{align*}
\end{proof}


\begin{lem}
	\label{l2}
	If  $\Gamma=\{ V(\Gamma),E(\Gamma),r,s \}$ is a finite, connected, directed graph such that \vspace{0.1cm}
	\begin{itemize}
		\item[(H 2.1)] there exists some $ h \in E(\Gamma)$ such that $ s^{-1}(r(h)) \neq \emptyset $, \vspace{0.1cm}
		\item[(H 2.2)] there exists $ e \in E(\Gamma)$ such that  $ s^{-1}(r(e))=\emptyset .$
	\end{itemize}
	\vspace{0.1cm} Then $q_{eh}=0$.
\end{lem}

\begin{proof}
	
	At $r(h)$, similarly we have
	\begin{align*}
		& \sum_{g\in E(\Gamma)} S_{g}^{*}S_{g}\otimes q_{gh}^{*}q_{gh}= \sum_{k,l\in E(\Gamma)}S_{k}S_{l}^{*} \otimes \left( \sum_{\{f:s(f)=r(h)\}} q_{kf}q_{lf}^{*}\right) \\      
		\Rightarrow  & \sum_{g:r(g)=r(e)} p_{r(e)}\otimes q_{gh}^{*}q_{gh} + \sum_{r(g)\neq r(e)} p_{r(g)} \otimes q_{gh}^{*}q_{gh} = \sum_{k,l\in E(\Gamma)}S_{k}S_{l}^{*} \otimes \left( \sum_{\{f:s(f)=r(h)\}} q_{kf}q_{lf}^{*}\right) 
	\end{align*} 
	Now, multiply both sides of the last equation by $(p_{r(e)} \otimes 1)$ from left and use the relation $p_{r(e)}=S_{e}^{*}S_{e}$
	\begin{equation}
		\sum_{r(g)=r(e)} p_{r(e)}\otimes q_{gh}^{*}q_{gh} + \sum_{r(g)\neq r(e)} p_{r(e)}p_{r(g)} \otimes q_{gh}^{*}q_{gh} = 
		\sum_{k,l\in E(\Gamma)}S_{e}^{*}S_{e}S_{k}S_{l}^{*} \otimes \left( \sum_{\{f:s(f)=r(h)\}} q_{kf}q_{lf}^{*}\right)  \label{l2.1}
	\end{equation}
	
	Observe that, $e$ can't be a loop ( as, $ s^{-1}(r(e))=\emptyset $ ). Therefore, $ S_{e}S_{k}=0 $ for all $ k\in E(\Gamma) $ [\text{by (H 2.2)} and Proposition 2.2 (iii)]. Also, orthogonality of $ \{p_{v}: v \in V(\Gamma)\} $ together with  \eqref{l2.1} implies
	\begin{align*}
		& \sum_{r(g)=r(e)} p_{r(e)}\otimes q_{gh}^{*}q_{gh}=0 \\
		\Rightarrow  &~~  q_{gh}^{*}q_{gh}=0 \text{ for all } g ~\text{such that} ~ r(g)=r(e)\\
		\Rightarrow  &~~  q_{eh} =0 .
	\end{align*}
	
\end{proof}


\begin{lem}
	\label{l3}
	If  $\Gamma=\{ V(\Gamma),E(\Gamma),r,s \}$ is a finite, connected, directed graph satisfying the following conditions:\vspace{0.1cm}
	\begin{itemize}
		\item[(H 3.1)] there exists some $ h \in E(\Gamma)$ such that $ s^{-1}(r(h)) \neq \emptyset $, \vspace{0.1cm}
		\item[(H 3.2)] $ e \in E(\Gamma)$ be a loop such that either $ r^{-1}(s(e))=\{e\} $ or $q_{gh}=0 $ for all $ g \in r^{-1}(s(e))-\{e\} \neq \emptyset. $
	\end{itemize}
	\vspace{0.1cm} Then $q_{ef}=0 $ for all $f \in s^{-1}(r(h)) \Leftrightarrow q_{eh}=0 $.\\ 
\end{lem}

\begin{proof}
	
	At $r(h)$, similarly one can write
	\begin{align*}
		& \sum_{g\in E(\Gamma)} S_{g}^{*}S_{g}\otimes q_{gh}^{*}q_{gh}= \sum_{k,l\in E(\Gamma)}S_{k}S_{l}^{*} \otimes \left( \sum_{\{f:s(f)=r(h)\}} q_{kf}q_{lf}^{*}\right) 
	\end{align*}	
	
	Now, multiplying both sides of the above equation by $(S_{e}^{*}\otimes 1)$ from left and $(S_{e}\otimes 1)$ from right, we get\\
	\begin{equation*}
		\sum_{g\in E(\Gamma)} S_{e}^{*}S_{g}^{*}S_{g}S_{e}\otimes q_{gh}^{*}q_{gh}= \sum_{k,l\in E(\Gamma)}S_{e}^{*}S_{k}S_{l}^{*}S_{e} \otimes \left( \sum_{\{f:s(f)=r(h)\}} q_{kf}q_{lf}^{*}\right) 
	\end{equation*}
	\begin{equation}
		\Rightarrow S_{e}^{*}S_{e}^{*}S_{e}S_{e}\otimes q_{eh}^{*}q_{eh}=S_{e}^{*}S_{e}S_{e}^{*}S_{e}\otimes \left( \sum_{\{f:s(f)=r(h)\}} q_{ef}q_{ef}^{*} \right) ~~~~[\text{ by (H 3.2) }]  \label{l3.1}
	\end{equation}
	
	Since $e$ is a loop,
	\begin{equation*}
		S_{e}^{*}S_{e}=S_{e}S_{e}^{*}+\sum_{\{f: ~ s(f)=s(e),~ f \neq  e\}}S_{f}S_{f}^{*}
	\end{equation*}	
	
	Multiplying both sides of the above equation by $S_{e}^{*}$ from left and $	S_{e}$ from right, we get
	
	\begin{align*}
		&~~~~~~~~~ S_{e}^{*}S_{e}^{*}S_{e}S_{e}=S_{e}^{*}S_{e}S_{e}^{*}S_{e} +  \sum_{\{f: ~ s(f)=s(e),~ f \neq  e\}}S_{e}^{*}S_{f}S_{f}^{*}S_{e}\\
		& \Rightarrow S_{e}^{*}S_{e}^{*}S_{e}S_{e}=S_{e}^{*}S_{e}S_{e}^{*}S_{e} = S_{e}^{*}S_{e}~~[ \text{ Since } S_{e}^{*}S_{f}=0 \text{ for } e\neq f ~]
	\end{align*}
	Hence using equation \eqref{l3.1},
	\begin{align*}
		S_{e}^{*}S_{e}\otimes q_{eh}^{*}q_{eh}=S_{e}^{*}S_{e} \otimes \left( \sum_{\{f:s(f)=r(h)\}} q_{ef}q_{ef}^{*} \right) 
	\end{align*}
	Therefore,
	\begin{equation*}
		q_{eh}^{*}q_{eh}=0 \Leftrightarrow \left( \sum_{\{f:s(f)=r(h)\}} q_{ef}q_{ef}^{*} \right)=0
	\end{equation*}
	i.e. 
	\begin{equation*}
		q_{eh}=0 \Leftrightarrow  q_{ef}=0 ~\text{for all}~ f \in s^{-1}(r(h)).
	\end{equation*}
	
\end{proof}
\vspace{1cm}

\noindent Before going to the proof of Proposition \ref{P1} and Proposition \ref{P2}, we define the following for our convenience:
$$Q_{ef}:=q_{ef}^{*}q_{ef}.$$ \vspace{0.1cm}

\noindent Moreover, we state some facts which will be used repeatedly in the proof of our propositions.\\

\noindent \textbf{(Fact 1)} $Q_{ef}=0 \Leftrightarrow q_{ef}=0.$\\ 

\noindent \textbf{(Fact 2)} Since $U^{t}=(q_{ef})_{e,f \in E(\Gamma)}^{t}$ is unitary, for each fixed $g$
\begin{align*}
	& ~~\sum_{k \in E(\Gamma)} q_{gk}^{*}q_{gk}=1  \\
	\Leftrightarrow & ~~ \sum_{k \in E(\Gamma)} Q_{gk}=1.
\end{align*}
\textbf{(Fact 3)} for each $g\in E(\Gamma)$
\begin{equation}
	Q_{gg}=1 \Leftrightarrow  Q_{gk}=0 \text{ for all } k \neq g  ~~~~~~~~~~~~[\text{using (Fact 2)}] \label{1} .  
\end{equation}
\vspace{0.1cm}
\textbf{(Fact 4)} If $q_{ef}=0$, then applying antipode $\kappa$ on both sides, one can conclude that $q_{fe}=0$.\\[0.2cm]

\begin{prop}
	\label{P1}
	Let $\Gamma \in \mathcal{G} $ such that  $ a_{11}=1 $, i.e. there exists a loop at vertex 1. Then $ Q_{\tau}^{Lin} \cong   \left(\mystar\limits_{|E(\Gamma)|} C(S^1),  \Delta\right)$.
\end{prop}\

\begin{proof}
	Note that $ {a_{11}=1} $, i.e. there exists a loop at vertex 1, by the hypothesis of the proposition.\\
	
	\noindent Our main idea for proving the proposition is the following:\\
	
	\noindent \textbf{\underline{Strategy:}} \textbf{ For each vertex $\mathbf{ u\in \{1,2,...n\} }$, we will show that $ \mathbf{q_{ee}^{*}q_{ee}=1 }$ for all $\mathbf{ e \in E(\Gamma)}$ with $ \mathbf{r(e)=u}$. This will be shown by using induction.} \\
	
	\noindent Let, 
	$\mathcal{S}$($u$): $ q_{ee}^{*}q_{ee}=1 $ for all $e \in E(\Gamma)$ such that $ r(e)=u $. Hence, our strategy is precisely to show that $\mathcal{S}$($u$) is true for all $ u\in \{1,2,...n\}$.\\
	
	\noindent Since a graph $\Gamma \in \mathcal{G} $ has no multiple edges, we denote the edge joining $i$ to $j$ by $e_{ij}$ i.e. $s(e_{ij})=i$ and $r(e_{ij})=j$.\\ 
	
	\noindent If $\Gamma\in \mathcal{G}$ with $E(\Gamma)=\{e_{11}\}$ then clearly $ Q_{\tau}^{Lin} \cong C(S^1)$. Therefore, we may assume that $\Gamma\in \mathcal{G}$ with $|E(\Gamma)|>1$. This implies $|V(\Gamma)|=n>1$. { \bf Now, our goal is to show that}  $\mathbf{\mathcal{S}(1)}$ \textbf{is true}.\\
	Since $\sum_{i=1}^{n} p_{i}=1$,
	\begin{align}
		S_{e_{11}}^{*}S_{e_{11}}+S_{e_{12}}^{*}S_{e_{12}}+S_{e_{23}}^{*}S_{e_{23}}+ \cdots +S_{e_{(n-1)n}}^{*}S_{e_{(n-1)n}}=1  \label{imp1.1}
	\end{align}
	\begin{align}
		\Rightarrow & \sum_{e_{kl}\in E(\Gamma)} S_{e_{kl}}^{*}S_{e_{kl}} \otimes (Q_{e_{kl}e_{11}}+Q_{e_{kl}e_{12}}+Q_{e_{kl}e_{23}}+ \cdots +Q_{e_{kl}e_{(n-1)n}})= 1\otimes 1 \\
		\Rightarrow & \sum_{u=1}^{n} p_{u} \otimes \left\lbrace  \sum_{r(e_{kl})=u}  (Q_{e_{kl}e_{11}}+Q_{e_{kl}e_{12}}+Q_{e_{kl}e_{23}}+ \cdots +Q_{e_{kl}e_{(n-1)n}})\right\rbrace = 1 \otimes 1 = \left( \sum_{u=1}^{n} p_{u} \right)  \otimes 1 \\
		\Rightarrow & \sum_{u=1}^{n} p_{u} \otimes \left\lbrace  \sum_{r(e_{kl})=u}  (Q_{e_{kl}e_{11}}+Q_{e_{kl}e_{12}}+Q_{e_{kl}e_{23}}+ \cdots +Q_{e_{kl}e_{(n-1)n}})-1 \right\rbrace = 0.
	\end{align}
	Since $ \{p_{u}\}_{u=1}^{n} $ is a linearly independent set, therefore for each $ u \in\{1,2,...n\} $
	\begin{equation}
		\sum_{r(e_{kl})=u}  (Q_{e_{kl}e_{11}}+Q_{e_{kl}e_{12}}+Q_{e_{kl}e_{23}}+\cdots +Q_{e_{kl}e_{(n-1)n}})=1   \label{imp1.2}
	\end{equation}

	For $u=1$,\\
	\begin{align*}
		&~~(Q_{e_{11}e_{11}}+Q_{e_{11}e_{12}}+Q_{e_{11}e_{23}}+ \cdots +Q_{e_{11}e_{(n-1)n}})= 1 = \sum_{e_{kl}\in E(\Gamma)} Q_{e_{11}e_{kl}} \\
		\Rightarrow &~~ Q_{e_{11}e_{kl}}=0 ~ \text{for all} ~ e_{kl}\in E(\Gamma) \text{ but }(k,l)\notin \{(1,1),(1,2),...,(n-1,n)\} \\
		\Rightarrow &~~ q_{e_{11}e_{kl}}=0 ~ \text{for all} ~ e_{kl}\in E(\Gamma) \text{ but }(k,l)\notin \{(1,1),(1,2),...,(n-1,n)\}.
	\end{align*}
	Now, it remains to show that $q_{e_{11}e_{kl}}=0 ~ \text{for all} ~ (k,l)\in \{(1,2),...,(n-1,n)\} $.\\
	To show this, we discuss the cases separately depending on the existence of a loop at vertex $n$.\\
	$ \bullet $  {\bf Assume  $\mathbf{a_{nn}=1}$, i.e.  there exists a loop at vertex $n$}.\\[0.1cm]
	We already have  $ q_{e_{11}e_{nn}}=0 $.\\
	Note that $e_{11}$ is a loop with $ r^{-1}(s(e_{11}))=\{e_{11}\} $ and $ e_{(n-1)n} \in E(\Gamma)$ such that $ s^{-1}(r(e_{(n-1)n})) \neq \emptyset $. Therefore by Lemma \ref{l3}, $ q_{e_{11}e_{(n-1)n}}=0 $ (as $ q_{e_{11}e_{nn}}=0 $).  \\
	
	\noindent $\bullet$ {\bf Assume $ \mathbf{a_{nn}=0} $, i.e. there does not exist any loop at vertex $n$}.\\[0.1cm]
	$e_{(n-1)n} \in E(\Gamma)$ with $s^{-1}(r(e_{(n-1)n})) = \emptyset $ and $ e_{11} \in E(\Gamma)$ such that $ s^{-1}(r(e_{11})) \neq \emptyset $. Therefore by Lemma \ref{l2}, $ q_{e_{(n-1)n}e_{11}}=0 $ 
	i.e. $q_{e_{11}e_{(n-1)n}}=0.$  \\
	
	Now, again $e_{(n-2)(n-1)}$ and $ e_{11} $ satisfy the hypothesis \textit{(H 3.1)} and \textit{(H 3.2)} respectively. Moreover, $q_{e_{11}f}=0 $ for all $f \in s^{-1}(r(e_{(n-2)(n-1)}))$. Therefore, $q_{e_{11}e_{(n-2)(n-1)}}=0 $. \\
	By repeated applications of Lemma \ref{l3}, we also get $ q_{e_{11}e_{kl}}=0 ~ \text{for all} ~ (k,l)\in \{(1,2),...,(n-1,n)\} $. Therefore, $ q_{e_{11}e_{kl}}=0 $ for all $ (k,l)\neq (1,1) $. In other words, $ q_{e_{11}e_{11}}^{*}q_{e_{11}e_{11}}=1 $. \\
	Hence, we can say that $ q_{ee}^{*}q_{ee}=1 $ for all $e \in E(\Gamma)$ with $ r(e)=1 $, i.e. 
	$\mathbf{\mathcal{S}(1)}$ \textbf{is true}.\\
	
	
	Let us assume that $\mathcal{S}(i)$ is true for all $ i \in \{1,2,...,(t-1)\}$.
	{\bf We will show that  $\mathcal{S}(t)$ is also true.}\\
	
	\textbf{(Case:1)} {\bf  Assume $ \mathbf{a_{tt}=1} $ i.e. there is a loop at vertex $t$. } \\
	We can rewrite Equation $ \eqref{imp1.1} $ just by replacing the $t^{th}$ term 
	\begin{align*}
		S_{e_{11}}^{*}S_{e_{11}}+S_{e_{12}}^{*}S_{e_{12}}+S_{e_{23}}^{*}S_{e_{23}}+ \cdots +S_{e_{(t-2)(t-1)}}^{*}S_{e_{(t-2)(t-1)}}+S_{e_{tt}}^{*}S_{e_{tt}}+ \cdots +S_{e_{(n-1)n}}^{*}S_{e_{(n-1)n}}=1
	\end{align*}
	and a similar computation will give us  
	\begin{align*}
		\sum_{r(e_{kl})=u}  (Q_{e_{kl}e_{11}}+Q_{e_{kl}e_{12}}+Q_{e_{kl}e_{23}}+\cdots +Q_{e_{kl}e_{(t-2)(t-1)}}+Q_{e_{kl}e_{tt}}+ \cdots + Q_{e_{kl}e_{(n-1)n}})=1 ~~~~\text{ for all }u \in \{1,2,...n\} 
	\end{align*}
	
	In particular, putting $u=t$;
	\begin{align*}
		& \sum_{r(e_{kl})=t}  (Q_{e_{kl}e_{11}}+Q_{e_{kl}e_{12}}+Q_{e_{kl}e_{23}}+\cdots +Q_{e_{kl}e_{(t-2)(t-1)}}+Q_{e_{kl}e_{tt}}+ \cdots + Q_{e_{kl}e_{(n-1)n}})=1\\
	\end{align*}
	
	Since by the induction hypothesis $ Q_{e_{kl}e_{11}}=Q_{e_{kl}e_{12}}=...=Q_{e_{kl}e_{(t-2)(t-1)}}=0 $ for all $e_{kl}\in E(\Gamma)$ such that $ r(e_{kl})=t $,
	
	\begin{align}
		\sum_{r(e_{kl})=t} (Q_{e_{kl}e_{tt}}+ \cdots + Q_{e_{kl}e_{(n-1)n}})=1  \label{imp1.3}
	\end{align}
	\begin{equation}
		\Rightarrow  Q_{e_{tt}e_{tt}}+Q_{e_{tt}e_{t(t+1)}}+ \cdots + Q_{e_{tt}e_{(n-1)n}}+\\
		\sum_{r(e_{kl})=t,(k,l)\neq (t,t)} (Q_{e_{kl}e_{tt}}+ \cdots + Q_{e_{kl}e_{(n-1)n}}) =1 \label{imp1.4}
	\end{equation}
	
	For a fix $ e_{kl} \in E(\Gamma)$ with  $r(e_{kl})=t$ such that $ (k,l)\neq (t,t) $, by the induction hypothesis $ q_{gg}^{*}q_{gg}=1 $ for all $ g \in r^{-1}(s(e_{kl})) $. This implies  $q_{ge_{(t-1)t}}=0$ for all $ g \in r^{-1}(s(e_{kl}))  $ [ using Equation $\eqref{1}$ ]. Therefore, by Lemma \ref{l1}, $ q_{e_{kl}e_{tt}}=q_{e_{kl}e_{t(t+1)}}=0.$ \\
	Similarly, for each $m \in \{t,(t+1),...,(n-1)\}$; $e_{m(m+1)} \in E(\Gamma) $ satisfies \textit{(H 1.1)} and also $e_{kl}$ satisfies \textit{(H 1.2)}. Therefore, by repeated applications of Lemma \ref{l1} (for each $m$),  $  q_{e_{kl}e_{(t+1)(t+2)}}=...= q_{e_{kl}e_{(n-1)n}}=0 $.	 \\
	Therefore,
	\begin{multline}
		q_{e_{kl}e_{tt}}=q_{e_{kl}e_{t(t+1)}}=...= q_{e_{kl}e_{(n-1)n}}=0 ~ \text{ for each } ~ e_{kl}\in E(\Gamma) ~ \text{ with } ~ r(e_{kl})=t ~ \text{ such that } ~ (k,l)\neq (t,t) \label{*}
	\end{multline}   
	
	These imply that $ Q _{e_{tt}e_{ij}}=0 $ for all $ (i,j) \notin \{(t,t),(t,t+1),...,(n-1,n)\} $ [ by Equation $\eqref{imp1.4}$ ].\\

	$\bullet$ {\bf Assume $ \mathbf{a_{nn}=1} $, i.e. there exists a loop at n.}\\
	From Equation $\eqref{*}$,  $ e_{tt} \in E(\Gamma)$ is a loop such that $ q_{ge_{(n-1)n}}=0 $ for all $ g \in r^{-1}(s(e_{tt}))-\{e_{tt}\} \neq \emptyset $ for $t>1$. 
	Therefore, \textit{(H 3.1)}, \textit{(H 3.2)} are satisfied by $e_{(n-1)n}$ and $e_{tt} \in E(\Gamma)$ respectively. Moreover, we have $q_{e_{tt}e_{nn}}=0$. This implies that $ q_{e_{tt}e_{(n-1)n}}=0 $ by Lemma \ref{l3}.\\

	$\bullet$ {\bf Assume $ \mathbf{a_{nn}=0} $, i.e. there does not exist any loop at n.}\\ 
	Observe that $ e_{tt} \in E(\Gamma)$ satisfies \textit{(H 2.1)} and $ e_{(n-1)n} \in E(\Gamma)$ satisfies \textit{(H 2.2)} because $ s^{-1}(n)= \emptyset $. Therefore by Lemma \ref{l2}, $ q_{e_{tt}e_{(n-1)n}}=0 $.\\
	
	For $t>1$, for all $ g \in r^{-1}(s(e_{tt}))-\{e_{tt}\} \neq \emptyset $, we have $q_{ge_{tt}}=q_{ge_{m(m+1)}}=0 $ for all $ m \in \{t,t+1,...,n-2\} $ [ by Equation \eqref{*} ]. Therefore, $ q_{e_{tt}e_{(n-1)n}}=0 \Rightarrow  q_{e_{tt}e_{(n-2)(n-1)}}=0 $. Similarly, at vertex $(n-3)$, we have already got that $ q_{e_{tt}f}=0$ for all $ f \in s^{-1}(n-3) $. Hence, again applying Lemma \ref{l3}, $ q_{e_{tt}e_{(n-3)(n-2)}}=0 $ and continue this process upto $ e_{t(t+1)} .$ \\
	
	Therefore, from the above discussions we get $ q _{e_{tt}e_{ij}}=0 $ for all $ (i,j) \in \{(t,t+1),...,(n-1,n)\} $\\
	Hence, $ q_{e_{tt}e_{kl}}=0 $ for all $ (k,l)\neq (t,t) $. In other words, $ q_{e_{tt}e_{tt}}^{*}q_{e_{tt}e_{tt}}=1 $. \\
	
	
	For $ j\neq t$,
	we can again rewrite Equation $\eqref{imp1.1}$ just by replacing $t^{th}$ term by $ S_{e_{jt}}^{*}S_{e_{jt}} $. So,
	\begin{align*}
		S_{e_{11}}^{*}S_{e_{11}}+S_{e_{12}}^{*}S_{e_{12}}+S_{e_{23}}^{*}S_{e_{23}}+ \cdots +S_{e_{(t-2)(t-1)}}^{*}S_{e_{(t-2)(t-1)}}+S_{e_{jt}}^{*}S_{e_{jt}}+ \cdots +S_{e_{(n-1)n}}^{*}S_{e_{(n-1)n}}=1
	\end{align*}
	and a similar computation will give us  
	\begin{align*}
		& \sum_{r(e_{kl})=t}  (Q_{e_{kl}e_{11}}+Q_{e_{kl}e_{12}}+Q_{e_{kl}e_{23}}+\cdots +Q_{e_{kl}e_{(t-2)(t-1)}}+Q_{e_{kl}e_{jt}}+ \cdots + Q_{e_{kl}e_{(n-1)n}})=1.\\
	\end{align*}
	
	\noindent Using the induction hypothesis same as before one can get that $ Q_{e_{kl}e_{11}}=Q_{e_{kl}e_{12}}=...=Q_{e_{kl}e_{(t-2)(t-1)}}=0 $ for all $e_{kl}\in E(\Gamma)$ such that $ r(e_{kl})=t $. Therefore for each $ j\neq t $ such that $e_{jt}\in E(\Gamma)$,
	
	\begin{align}
		\sum_{r(e_{kl})=t} (Q_{e_{kl}e_{jt}}+ \cdots + Q_{e_{kl}e_{(n-1)n}})=1  \label{imp1.5}
	\end{align}
	\begin{align*}
	\Rightarrow & ~~ (Q_{e_{tt}e_{jt}}+Q_{e_{tt}e_{t(t+1)}}+ \cdots + Q_{e_{tt}e_{(n-1)n}})+\\
	& ~~ (Q_{e_{jt}e_{jt}}+Q_{e_{jt}e_{t(t+1)}}+ \cdots + Q_{e_{jt}e_{(n-1)n}})+\\
	& ~~ \sum_{r(e_{kl})=t,~(k,l)\neq (j,t)(t,t)} (Q_{e_{kl}e_{jt}}+ \cdots + Q_{e_{kl}e_{(n-1)n}}) =1.
	\end{align*}
	
	As $ j\neq t$ and $Q_{e_{tt}e_{tt}}=1 $, $(Q_{e_{tt}e_{jt}}+Q_{e_{tt}e_{t(t+1)}}+ \cdots + Q_{e_{tt}e_{(n-1)n}})=0 $. Hence, we have arrived at an equation which is almost similar to Equation \eqref{imp1.4}. Now, arguing exactly like the previous one, one can show that all the other terms except $Q_{e_{jt}e_{jt}}$ of the above equation are 0. Therefore, $Q_{e_{jt}e_{jt}}=1 $ for all $ j \neq t $.
	We have already seen that $Q_{e_{tt}e_{tt}}=1 $.\\
	So, Combining these facts we can conclude that $\mathbf{\mathcal{S}(t)}$ \textbf{is also true}.\\
	

	\textbf{(Case:2)}  {\bf Assume $ \mathbf{a_{tt}=0} $, i.e. there does not exist any loop at $t$.} \\
	Let, $m$ be such that $e_{mt} \in E(\Gamma)$.
	Rewriting Equation $\eqref{imp1.1}$ just replacing $t^{th}$ term by $ S_{e_{mt}}^{*}S_{e_{mt}} $. So,
	\begin{align*}
		S_{e_{11}}^{*}S_{e_{11}}+S_{e_{12}}^{*}S_{e_{12}}+S_{e_{23}}^{*}S_{e_{23}}+ \cdots +S_{e_{(t-2)(t-1)}}^{*}S_{e_{(t-2)(t-1)}}+S_{e_{mt}}^{*}S_{e_{mt}}+ \cdots +S_{e_{(n-1)n}}^{*}S_{e_{(n-1)n}}=1
	\end{align*}
	and a similar computation will give us  
	\begin{align*}
		& \sum_{r(e_{kl})=t}  (Q_{e_{kl}e_{11}}+Q_{e_{kl}e_{12}}+Q_{e_{kl}e_{23}}+\cdots +Q_{e_{kl}e_{(t-2)(t-1)}}+Q_{e_{kl}e_{mt}}+ \cdots + Q_{e_{kl}e_{(n-1)n}})=1\\
	\end{align*}
	
	\noindent and again using the induction hypothesis, one can get that $ Q_{e_{kl}e_{11}}=Q_{e_{kl}e_{12}}=...=Q_{e_{kl}e_{(t-2)(t-1)}}=0 $ for all $e_{kl}\in E(\Gamma)$ such that $ r(e_{kl})=t $. Therefore,
	\begin{align*}
		& \sum_{r(e_{kl})=t} (Q_{e_{kl}e_{mt}}+ \cdots + Q_{e_{kl}e_{(n-1)n}})=1\\
		\Rightarrow & (Q_{e_{mt}e_{mt}}+Q_{e_{mt}e_{t(t+1)}}+ \cdots + Q_{e_{mt}e_{(n-1)n}})+ \sum_{r(e_{kl})=t,~(k,l)\neq (m,t)} (Q_{e_{kl}e_{mt}}+ \cdots + Q_{e_{kl}e_{(n-1)n}}) =1.           
	\end{align*}
	
	Now arguing exactly the same as before, after Equation \eqref{imp1.4} in \textbf{(Case:1)}, we can show that all the other terms except $Q_{e_{mt}e_{mt}}$ of the above equation are 0. Therefore, $Q_{e_{mt}e_{mt}}=1 $ for each $m$ such that $e_{mt} \in E(\Gamma)$.
	Therefore, in this case also  $\mathbf{\mathcal{S}(t)}$ \textbf{is true}.
	Hence, we are done.
\end{proof}

\begin{prop}
	\label{P2}
	Let $\Gamma \in \mathcal{G} $ such that  $ a_{11}=0 $, i.e. there does not exist any loop at vertex 1. Then $ Q_{\tau}^{Lin} \cong  \left(\mystar\limits_{|E(\Gamma)|} C(S^1),  \Delta \right)$.\\
\end{prop}
\begin{proof}
	By the hypothesis of the proposition, observe that  $ {a_{11}=0} $, i.e. there does not exist any loop at vertex 1. \\
	Firstly, if $\Gamma\in \mathcal{G}$ with $|E(\Gamma)|=1$ then $|V(\Gamma)|=2$ and $E(\Gamma)=\{e_{12}\}$. In this case, it is trivial to observe that $Q_{\tau}^{Lin} \cong  C(S^1)$. So, again we may assume that $|E(\Gamma)|>1$.\\
	From now on the strategy of the proof is exactly the same as what we have used in Proposition \ref{P1}.\\
	
	\noindent Now, for any $ e_{1x} \in E(\Gamma) $ satisfies \textit{(H 1.2)} [ as $ r^{-1}(1)= \emptyset $ ] and every $e_{uv} \in E(\Gamma)$ with $ v \neq n $ satisfies \textit{(H 1.1)}. Therefore, using Lemma \ref{l1} repeatedly, we can conclude that $ q_{e_{1x}e_{kl}}=0 $ for all $ e_{kl} \in E(\Gamma) $ such that $k \neq 1$.\\
	
	Since $$ \sum_{i=1}^{n} p_{i}=1 $$
	\begin{align}
		\Rightarrow  \left( \sum_{e_{1x}\in E(\Gamma)}S_{e_{1x}}S_{e_{1x}}^{*} \right) +S_{e_{12}}^{*}S_{e_{12}}+S_{e_{23}}^{*}S_{e_{23}}+ \cdots +S_{e_{(n-1)n}}^{*}S_{e_{(n-1)n}}=1 \label{imp2.1}
	\end{align}
	
	Applying $\alpha$ on both sides of the above equation, we get
	\begin{equation}
		\sum_{e_{1x}\in E(\Gamma)}\left( \sum_{e_{ij},e_{kl}\in E(\Gamma)} S_{e_{ij}}S_{e_{kl}}^{*}\otimes q_{e_{ij}e_{1x}}q_{e_{kl}e_{1x}}^{*} \right)
		+ \sum_{e_{kl}\in E(\Gamma)} S_{e_{kl}}^{*}S_{e_{kl}} \otimes (Q_{e_{kl}e_{12}}+Q_{e_{kl}e_{23}}+ \cdots +Q_{e_{kl}e_{(n-1)n}})= 1\otimes 1  \label{imp}
	\end{equation}
	
	Now, we focus on the term $ \sum_{e_{1x}\in E(\Gamma)}\left( \sum_{e_{ij},e_{kl}\in E(\Gamma)} S_{e_{ij}}S_{e_{kl}}^{*}\otimes q_{e_{ij}e_{1x}}q_{e_{kl}e_{1x}}^{*} \right)  $.\\
	
	Since $ q_{e_{1x}e_{kl}}=0 $ for all $ e_{kl} \in E(\Gamma) $ such that $k \neq 1$, $ q_{e_{ij}e_{1x}}q_{e_{kl}e_{1x}}^{*}=0 $ whenever at least one of $ i,k \neq 1 $. It means if $ q_{e_{ij}e_{1x}}q_{e_{kl}e_{1x}}^{*} \neq 0 $ then it must be of the form $ q_{e_{1j}e_{1x}}q_{e_{1l}e_{1x}}^{*} $.\\
	Moreover,  $ S_{e_{ij}}S_{e_{kl}}^{*} \neq 0 \Leftrightarrow r(e_{ij})=r(e_{kl}) \Leftrightarrow j=l $. Therefore, if $ S_{e_{ij}}S_{e_{kl}}^{*}\otimes q_{e_{ij}e_{1x}}q_{e_{kl}e_{1x}}^{*} \neq 0 $ then it must be of form $ S_{e_{1j}}S_{e_{1j}}^{*}\otimes q_{e_{1j}e_{1x}}q_{e_{1j}e_{1x}}^{*} $. Hence,
	
	\begin{equation*}
		\sum_{e_{1x},e_{ij},e_{kl}\in E(\Gamma)} S_{e_{ij}}S_{e_{kl}}^{*}\otimes q_{e_{ij}e_{1x}}q_{e_{kl}e_{1x}}^{*}= \sum_{e_{1j}\in E(\Gamma)}S_{e_{1j}}S_{e_{1j}}^{*}\otimes \left(  \sum_{e_{1x}\in E(\Gamma)}q_{e_{1j}e_{1x}}q_{e_{1j}e_{1x}}^{*}\right). 
	\end{equation*}
	
	This together with Equation $ \eqref{imp} $ implies
	
	\begin{multline*}
	\sum_{e_{1j}\in E(\Gamma)}S_{e_{1j}}S_{e_{1j}}^{*}\otimes \left(  \sum_{e_{1x}\in E(\Gamma)}q_{e_{1j}e_{1x}}q_{e_{1j}e_{1x}}^{*}\right)+ \\    
	\sum_{e_{kl}\in E(\Gamma)} S_{e_{kl}}^{*}S_{e_{kl}} \otimes (Q_{e_{kl}e_{12}}+Q_{e_{kl}e_{23}}+ \cdots +Q_{e_{kl}e_{(n-1)n}})= 1\otimes 1 
	\end{multline*}
	\begin{multline*}
	\Rightarrow  \sum_{e_{1j}\in E(\Gamma)}S_{e_{1j}}S_{e_{1j}}^{*}\otimes \left(  \sum_{e_{1x}\in E(\Gamma)}q_{e_{1j}e_{1x}}q_{e_{1j}e_{1x}}^{*}\right) \\
	+  \sum_{u=2}^{n} p_{u} \otimes \left\lbrace  \sum_{r(e_{kl})=u}  (Q_{e_{kl}e_{12}}+Q_{e_{kl}e_{23}}+ \cdots +Q_{e_{kl}e_{(n-1)n}})\right\rbrace = \left( \sum_{u=1}^{n} p_{u} \right)  \otimes 1 \\
	\end{multline*}           
	\begin{multline*}
	\Rightarrow  \sum_{e_{1j}\in E(\Gamma)}S_{e_{1j}}S_{e_{1j}}^{*}\otimes \left(  \sum_{e_{1x}\in E(\Gamma)}q_{e_{1j}e_{1x}}q_{e_{1j}e_{1x}}^{*}\right) \\
	+  \sum_{u=2}^{n} p_{u} \otimes \left\lbrace  \sum_{r(e_{kl})=u}  (Q_{e_{kl}e_{12}}+Q_{e_{kl}e_{23}}+ \cdots +Q_{e_{kl}e_{(n-1)n}})-1\right\rbrace = p_{1}  \otimes 1 \\
	\end{multline*}
	Since $p_{1}=\sum_{e_{1j}\in E(\Gamma)}S_{e_{1j}}S_{e_{1j}}^{*}$,
	\begin{multline*}
	\sum_{e_{1j}\in E(\Gamma)}S_{e_{1j}}S_{e_{1j}}^{*}\otimes \left(  \sum_{e_{1x}\in E(\Gamma)}q_{e_{1j}e_{1x}}q_{e_{1j}e_{1x}}^{*} -1 \right) \\
	+  \sum_{u=2}^{n} p_{u} \otimes \left\lbrace  \sum_{r(e_{kl})=u}  (Q_{e_{kl}e_{12}}+Q_{e_{kl}e_{23}}+ \cdots +Q_{e_{kl}e_{(n-1)n}})-1\right\rbrace = 0 \\
	\end{multline*}
	For each fix $ i \in \{2,3,...n\} $, multiplying both sides of the above equation by $p_{i}\otimes 1 $ from left and using $ p_{i}=S_{e_{(i-1)i}}^{*}S_{e_{(i-1)i}} $, we get 
	\begin{multline*}
	\sum_{e_{1j}\in E(\Gamma)} S_{e_{(i-1)i}}^{*}S_{e_{(i-1)i}}S_{e_{1j}}S_{e_{1j}}^{*}\otimes \left(  \sum_{e_{1x}\in E(\Gamma)}q_{e_{1j}e_{1x}}q_{e_{1j}e_{1x}}^{*} -1 \right) \\
	+  \sum_{u=2}^{n}p_{i} p_{u} \otimes \left\lbrace  \sum_{r(e_{kl})=u}  (Q_{e_{kl}e_{12}}+Q_{e_{kl}e_{23}}+ \cdots +Q_{e_{kl}e_{(n-1)n}})-1\right\rbrace = 0 \\
	\end{multline*}	
	
	Since $ i \neq 1, S_{e_{(i-1)i}}S_{e_{1j}}=0 $ for all $j$ and orthogonality of $ \{p_{i}\}_{i=2,...,n} $ implies for each $u \in \{2,...,n\}$
	
	\begin{equation}
	\sum_{r(e_{kl})=u}  (Q_{e_{kl}e_{12}}+Q_{e_{kl}e_{23}}+ \cdots +Q_{e_{kl}e_{(n-1)n}})=1 \label{imp2.2}
	\end{equation}
	Now, the strategy is exactly the same as what we have done in Proposition \ref{P1}.
	Note that $r^{-1}(1)= \emptyset$ (as $ \Gamma \in \mathcal{G} $ with $a_{11}=0$). So, there is nothing to do with $\mathcal{S}(1)$. Thus,{ \bf  we will show that $\mathcal{S}(2)$ is true.}\\
	For each $y$ such that $ e_{y2}\in E(\Gamma)$ (i.e. the possible values of $y$ are 1 and 2), we can modify Equation $ \eqref{imp2.1} $ just by putting the value of $p_{2}$ as $S_{e_{y2}}^{*}S_{e_{y2}}$, we get
	\begin{align*}
		\left( \sum_{e_{1k}\in E(\Gamma)}S_{e_{1k}}S_{e_{1k}}^{*} \right) +S_{e_{y2}}^{*}S_{e_{y2}}+S_{e_{23}}^{*}S_{e_{23}}+ \cdots  +S_{e_{(n-1)n}}^{*}S_{e_{(n-1)n}}=1 
	\end{align*}
	and similarly, one can derive that
	\begin{equation}
		\sum_{r(e_{kl})=u} (Q_{e_{kl}e_{y2}}+Q_{e_{kl}e_{23}}+ \cdots +Q_{e_{kl}e_{(n-1)n}})=1 \label{e2,u=2} \\
	\end{equation}
	
	Now, observe that $e_{12} \in E(\Gamma)$ such that $r^{-1}(s(e_{12}))= \emptyset $ and $e_{m(m+1)} \in E(\Gamma)$ with $s^{-1}(r(e_{m(m+1)}))= s^{-1}(m+1) \neq \emptyset $ for all $m=1,2,...,(n-1)$. Hence, by Lemma \ref{l1}, $q_{e_{12}e_{kl}}=0$ for all $k \in \{2,3,...,n\}$. In particular, $q_{e_{12}e_{23}}=q_{e_{12}e_{34}}=...=q_{e_{12}e_{(n-1)n}}=0$. Moreover, note that $q_{e_{12}e_{22}}=0$ if $e_{22} \in E(\Gamma)$.\\
	
	\noindent Now, there will be two cases depending on the existence of a loop at vertex 2.\\
	{\bf (Case: 1) Assume $ \mathbf{a_{22}=1} $, i.e. there exists a loop at vertex 2.}\\
	Firstly, by putting $u=2$ and $y=2$ in Equation \eqref{e2,u=2} we have
	\begin{equation*}
		(Q_{e_{12}e_{22}}+Q_{e_{12}e_{23}}+ \cdots +Q_{e_{12}e_{(n-1)n}})+
		(Q_{e_{22}e_{22}}+Q_{e_{22}e_{23}}+ \cdots +Q_{e_{22}e_{(n-1)n}})=1. 
	\end{equation*}
	Since $Q_{e_{12}e_{22}}=Q_{e_{12}e_{23}}= \cdots =Q_{e_{12}e_{(n-1)n}}=0$, the above equation reduces to 
	\begin{equation*}
		(Q_{e_{22}e_{22}}+Q_{e_{22}e_{23}}+ \cdots +Q_{e_{22}e_{(n-1)n}})=1.
	\end{equation*}
	Therefore, $ Q_{e_{22}e_{ij}}=0$ for all $(i,j) \notin \{(2,2),(2,3),...,(n-1,n)\}$ [ using \textbf{(Fact 2)} ]. Now, our target is to show that $ Q_{e_{22}e_{23}}= \cdots =Q_{e_{22}e_{(n-1)n}}=0$. The arguments are exactly same as what we did after Equation \eqref{*} onwards in Proposition \ref{P1}. Hence, from the above equation, one will be able to show that $q_{e_{22}e_{kl}}=0$ for all $(k,l) \neq (2,2)$. Hence, $Q_{e_{22}e_{22}}=1$. \\
	Secondly, by putting $u=2$ and $y=1$ in Equation \eqref{e2,u=2} we have 
	\begin{equation*}
		(Q_{e_{12}e_{12}}+Q_{e_{12}e_{23}}+ \cdots +Q_{e_{12}e_{(n-1)n}})+
		(Q_{e_{22}e_{12}}+Q_{e_{22}e_{23}}+ \cdots +Q_{e_{22}e_{(n-1)n}})=1 .
	\end{equation*}
	We have already shown that all the terms except $Q_{e_{12}e_{12}}$ of the above equation are 0. Hence, $Q_{e_{12}e_{12}}=1$.\\
	
	\noindent {\bf (Case: 2) Assume $ \mathbf{a_{22}=0} $, i.e. there does not exist any loop at vertex 2.}\\
	In this case, Equation \eqref{e2,u=2} with $u=2$ will be in the form   
	\begin{equation*}
		(Q_{e_{12}e_{12}}+Q_{e_{12}e_{23}}+ \cdots +Q_{e_{12}e_{(n-1)n}})=1 .
	\end{equation*}
	Since $ q_{e_{12}e_{23}}=...=q_{e_{12}e_{(n-1)n}}=0 $, clearly $ Q_{e_{12}e_{12}}=1$.\\
	Combining these two cases, {\bf  $\mathcal{S}(2)$ is true.}\\  
	
	\noindent Let by the induction hypothesis again $\mathcal{S}(i)$ is true for $i=1,2,...,(t-1).$ Now, {\bf we will show that $\mathcal{S}(t)$ is also true.}\\
	
	\noindent Again, we have divided the proof into two cases depending on the existence of a loop at vertex $t$.\\
	\noindent  {\bf (Case: 1) Assume $ \mathbf{a_{tt}=1} $, i.e. there exists a loop at t.}\\
	For each $y$ such that $ e_{yt}\in E(\Gamma)$, we can modify Equation $ \eqref{imp2.1} $ just by putting the value of $p_{t}$ as $S_{e_{yt}}^{*}S_{e_{yt}}$, we get
	\begin{align*}
		\left( \sum_{e_{1k}\in E(\Gamma)}S_{e_{1k}}S_{e_{1k}}^{*} \right) +S_{e_{12}}^{*}S_{e_{12}}+S_{e_{23}}^{*}S_{e_{23}}+ \cdots +S_{e_{yt}}^{*}S_{e_{yt}}+ \cdots +S_{e_{(n-1)n}}^{*}S_{e_{(n-1)n}}=1 
	\end{align*}
	and a similar computation will give us
	\begin{equation*}
		\sum_{r(e_{kl})=u} (Q_{e_{kl}e_{12}}+Q_{e_{kl}e_{23}}+ \cdots + Q_{e_{kl}e_{yt}}+ \cdots +Q_{e_{kl}e_{(n-1)n}})=1. 
	\end{equation*}
	Put $u=t$ on the above equation.
	
	\begin{equation*}
		\sum_{r(e_{kl})=t} (Q_{e_{kl}e_{12}}+Q_{e_{kl}e_{23}}+ \cdots +Q_{e_{kl}e_{yt}}+ \cdots +Q_{e_{kl}e_{(n-1)n}})=1. 
	\end{equation*}
	By the induction hypothesis, $ Q_{e_{kl}e_{12}}=Q_{e_{kl}e_{23}}=...=Q_{e_{kl}e_{(t-2)(t-1)}}=0 $. Therefore for each $e_{yt} \in E(\Gamma)$, 
	\begin{equation}
		\sum_{r(e_{kl})=t} (Q_{e_{kl}e_{yt}}+ Q_{e_{kl}e_{t(t+1)}}+\cdots +Q_{e_{kl}e_{(n-1)n}})=1   \label{imp2.3}
	\end{equation}
	
	Now, we have arrived at the same situation as in Proposition \ref{P1}. Observe that Equation $ \eqref{imp2.3} $ with $y=t$ and Equation   $\eqref{imp2.3}$ for $y=j \neq t $ are exactly the same as Equations $ \eqref{imp1.3} $ and $\eqref{imp1.5}$ respectively. So arguing these two cases successively just like the previous one, we will be able to show that $Q_{e_{mt}e_{mt}}=1$ for each $ m $ such that $e_{mt} \in E(\Gamma)$.\\
	Therefore, {\bf $\mathcal{S}(t)$ is true for this case.}\\

	\noindent \textbf{(Case: 2)  Assume $ \mathbf{a_{tt}=0 }$, i.e. if there does not exist any loop at $\mathbf{t}$. }\\
	Proof is similar as Proposition \ref{P1}, \textbf{(Case 2)}. We just need to start with the equation
	\begin{align*}
		\left( \sum_{e_{1k}\in E(\Gamma)}S_{e_{1k}}S_{e_{1k}}^{*} \right) +S_{e_{12}}^{*}S_{e_{12}}+S_{e_{23}}^{*}S_{e_{23}}+ \cdots +S_{e_{mt}}^{*}S_{e_{mt}} + \cdots +S_{e_{(n-1)n}}^{*}S_{e_{(n-1)n}}=1 
	\end{align*}
	where $ m \in \{1,2,...,n\} $ be such that $ e_{mt} \in E(\Gamma) $.\\
	
	By a similar computation, we will reach at 
	\begin{equation*}
		\sum_{r(e_{kl})=t} (Q_{e_{kl}e_{mt}}+ Q_{e_{kl}e_{t(t+1)}}+\cdots +Q_{e_{kl}e_{(n-1)n}})=1   
	\end{equation*}
	and the rest of the proof is again similar to the arguments that we have used after Equation \eqref{imp1.4} and using that one can show that  $Q_{e_{mt}e_{mt}}=1$ for all $m$  with $ e_{mt} \in E(\Gamma) $.
	Hence, $\mathbf{\mathcal{S}(t)}$ \textbf{is also true}  for this case. Therefore, again we are done. \\ 
\end{proof}

\noindent Finally, we conclude our main theorem just by combining the above two propositions.
\begin{thm}
	\label{T1}
	Let $\Gamma=\{V(\Gamma), E(\Gamma), r, s\}$ be a finite, connected, directed graph satisfying the following graph theoretic properties : \textbf{(R1)} there does not exist any cycle of length $\geq$ 2, \textbf{(R2)} there exists a path of length $(|V(\Gamma)|-1)$ which consists all the vertices and \textbf{(R3)} given any two vertices (may not be distinct) there exists at most one edge joining them, then $(Q_{\tau}^{Lin}, \Delta)$  is CQG isomorphic to $\left( \mystar\limits_{|E(\Gamma)|} C(S^1),  \Delta\right) $,\\
	In other words, for any $\Gamma \in \mathcal{G} $, $ Q_{\tau}^{Lin} \cong $ $\left(\mystar\limits_{|E(\Gamma)|} C(S^1),  \Delta\right)$.
\end{thm} 
\begin{proof}
	Since $\Gamma \in \mathcal{G}$, we can divide the proof into two cases:  (i) $a_{11}=1$ and (ii)  $a_{11}=0$. In both cases, we are done because of  Proposition \ref{P1} and Proposition \ref{P2} respectively.
	
\end{proof}
\section{Counter examples}
In this section, we cook up some counter examples which help us to understand that this class of graphs is really significant. If we deviate slightly from any of the assumptions in \textbf{(R)}, i.e. we move away from the desired form of matrices that we have got in section 3, then quantum symmetry may not be isomorphic to $\left(\mystar\limits_{|E(\Gamma)|} C(S^1),  \Delta\right)$ with respect to the category  $\mathfrak{C}_{\tau}^{Lin}$. \\

(1) Let $\mathcal{K}_{2}$ be a complete graph with 2 vertices i.e. the adjacency matrix is of the form 
$\begin{bmatrix}
0 & 1 \\
1 & 0 \\
\end{bmatrix} $. Then the universal object $Q_{\tau}^{Lin}$ for $C^*(\mathcal{K}_{2})$ is isomorphic to $ (\mathcal{D}_{\varphi}(C(S^1)* C(S^1)),\Delta_{\varphi})$, where $\varphi$ is an order two automorphism such that $\varphi(z_{1})=z_{2}$ and $\varphi(z_{2})=z_{1}$ for two canonical generators $z_{1}, z_{2}$ of $C(S^1) * C(S^1)$ (for details see Theorem 5.4 from \cite{Mandal}). Observe that the adjacency matrix is not an upper triangular matrix (since the graph itself is a cycle of length 2, i.e. \textbf{(R1) has been violated}) and also the universal object $Q_{\tau}^{Lin}$ is not isomorphic to $C(S^{1})*C(S^{1})$. \\

\begin{center}
	\begin{figure}[htpb] \label{K2}
		\begin{tikzpicture}
		\draw[fill=black] (0,0) circle (2pt) node[anchor=east]{$v_{1}$};
		\draw[fill=black] (2,0) circle (2pt) node[anchor=west]{$v_{2}$};
		
		\draw[black,thick](0,0) edge[bend left= 60] node{\rmidarrow} node[above]{$e_{12}$} (2,0);
		\draw[black,thick](2,0) edge[bend left= 60] node{\lmidarrow} node[below]{$e_{21}$} (0,0);
		
		\end{tikzpicture}
		\caption{$\mathcal{K}_{2}$}
	\end{figure}
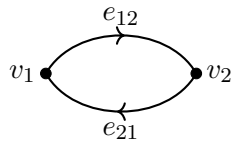
\end{center}

(2) Let $L_{1,1}$ be a graph whose adjacency matrix is given by $Id_{2 \times 2}$. Then the universal object $Q_{\tau}^{Lin}$ for  $C^*(L_{1,1})$ is isomorphic to $H_{2}^{\infty+}$ (see proposition 4.2 from \cite{Mandalkms}). In this case also the graph does not belong to $\mathcal{G}$, since \textbf{(R2) is violated} (i.e. super diagonal entry of the adjacency matrix of $L_{1,1}$ is 0). Note that quantum symmetry $Q_{\tau}^{Lin}$ is not isomorphic to $C(S^{1})*C(S^{1})$. \\

\begin{center}
	\begin{figure}[htpb] \label{L11}
		\begin{tikzpicture}
		\draw[fill=black] (0,0) circle (2pt) node[anchor=south]{$v_{1}$};
		\draw[fill=black] (2,0) circle (2pt) node[anchor=south]{$v_{2}$};

		\begin{scope}[decoration={markings,
			mark=at position 0.25 with {\arrow[black,thick]{<}},
		}]
		\draw[black,postaction={decorate}, thick](0,0.5) circle (0.5)  node[above=0.6]{$e_{11}$};
		\end{scope}

		\begin{scope}[decoration={markings,
			mark=at position 0.25 with {\arrow[black,thick]{<}},
		}]
		\draw[black,postaction={decorate}, thick](2,0.5) circle (0.5)  node[above=0.6]{$e_{22}$};
		\end{scope}
		
		\end{tikzpicture}
		\caption{$L_{1,1}$}
	\end{figure}
\end{center}

(3) Let $\Gamma$ be a finite, connected, directed graph whose adjacency matrix contains a natural number other than 0 and 1 (i.e. \textbf{(R3) is being dropped}). Assume that $E(\Gamma)=\{e_{i}:i=1,2,...,|E(\Gamma)|\}$ and $\{z_{i}:i=1,2,...,|E(\Gamma)|\}$ are the unitary operators which generate the universal $C^*$-algebra $\left(\mystar\limits_{|E(\Gamma)|} C(S^1),  \Delta \right)$ . Let $a_{ij}=n(v_{i},v_{j})>1 $ and without loss of generality $e_{1},e_{2} \in \{e \in E(\Gamma): s(e)=v_{i}, r(e)=v_{j}\}$. Then $\left( \mathcal{D}_{\varphi}\left(\mystar\limits_{|E(\Gamma)|} C(S^1)\right),\Delta_{\varphi}\right)$ always acts on $C^*(\Gamma)$ with respect to the automorphism $\varphi(z_{1})=z_{2}$, $\varphi(z_{2})=z_{1}$ and $ \varphi(z_{k})=z_{k}$ for $k \neq 1, 2 $ . Moreover, a faithful $\tau$ preserving action $\alpha:C^*(\Gamma) \to C^*(\Gamma) \otimes \left( \mathcal{D}_{\varphi}\left(\mystar\limits_{|E(\Gamma)|} C(S^1)\right),\Delta_{\varphi}\right)  $ is given by 
\begin{align}
	\alpha(S_{e_1}) &= S_{e_1} \otimes (z_{1},0)+S_{e_2} \otimes (0,z_{2}),\\
	\alpha(S_{e_2}) &= S_{e_2} \otimes (z_{2},0)+S_{e_1} \otimes (0,z_{1}),\\
	\alpha(S_{e_k}) &= S_{e_k} \otimes (z_{k},z_{k}) ~~~~ (\text{for }  k \neq 1,2 ).
\end{align}

Quantum SO(3) $(C(SO_{q}(3)))$, odd dimensional quantum real projective space $C(\mathbb{R}P_{q}^{2n-1})$ for $n>1$ and even dimensional quantum real projective space $C(\mathbb{R}P_{q}^{2n})$ for $n \geq 1$ can be viewed as graph $C^*$-algebra for $q \in [0,1)$.
More precisely, for $q \in [0,1)$, $C(\mathbb{R}P_{q}^{2n-1})$ (respectively $C(\mathbb{R}P_{q}^{2n}$)) is isomorphic to $C^*(L_{2n-1}^{(2)})$ (respectively $C^*(L_{2n}')$). Explicit descriptions of $C^*(L_{2n-1}^{(2)})$ and $C^*(L_{2n}')$ can be found in section 0 (Introduction), 4.2 and 5.2 of \cite{HSproj}. In particular, $C(SO_{q}(3)) \cong C(\mathbb{R}P_{q}^{3}) \cong C^*(L_{3}^{(2)}).$ From the discussion in the previous paragraph, one can conclude that the quantum symmetry $Q_{\tau}^{Lin}$ can not be isomorphic to $\left(\mystar\limits_{|E(\Gamma)|} C(S^1),  \Delta\right)$corresponding to these graphs mentioned above. \\
\begin{figure}[htpb] \label{L2'}
	
	\begin{tikzpicture}
	\draw[fill=black] (0,0) circle (2pt) node[anchor=east]{$v_{1}$};
	\draw[fill=black] (2,0) circle (2pt) node[anchor=west]{$v_{2}$};
	
	\draw[black,thick](-0.5,0) circle (-0.5) node[above=0.25]{\rmidarrow};
	\draw[black,thick](0,0) edge[bend left= 60] node{\rmidarrow} node[above=0.1]{$e_{1}$} (2,0);
	\draw[black,thick](2,0) edge[bend left= 60] node{\rmidarrow} node[below=0.1]{$e_{2}$} (0,0);
	
	\end{tikzpicture}
	\caption{$L_{2}'$}
\end{figure}

\begin{figure}[htpb]  \label{L3(2)}
	\begin{tikzpicture}
	\draw[fill=black] (0,0) circle (2pt) node[anchor=south]{$v_{1}$};
	\draw[fill=black] (2,0) circle (2pt) node[anchor=south]{$v_{2}$};
	
	\draw[black,thick](0,0.5) circle (0.5) node[above=0.25]{\rmidarrow};
	\draw[black,thick](2,0.5) circle (0.5) node[above=0.25]{\rmidarrow} ;
	
	\draw[black,thick] (0,0)-- node{\rmidarrow}  node[above]{$e_{1}$} (2,0);
	
	\draw[black,thick] (0,0) edge[bend left= -45] node{\rmidarrow}  node[below]{$e_{2}$} (2,0);
	
	\end{tikzpicture}
	\caption{$L_{3}^{(2)}$}  
	
\end{figure}

Moreover, converse of Theorem \ref{T1} is not true i.e. the quantum symmetry $Q_{\tau}^{Lin}$ for a graph $C^*$-algebra $C^*(\Gamma) $ isomorphic to $\left(\mystar\limits_{|E(\Gamma)|} C(S^1),  \Delta\right)$ does not imply the underlying graph must be in $\mathcal{G}$. More precisely, \textbf{(R1)}, \textbf{(R2)} are not necessary but \textbf{(R3)} is necessary for Theorem \ref{T1}.  \\

\begin{enumerate}
	\item[(A)] The condition \textbf{(R1)} is not  necessary. Consider the graph $\Gamma_{0}$ whose adjacency matrix is given by
	$ \begin{bmatrix}
	0&1&0\\
	0&0&1\\
	0&1&0\\
	\end{bmatrix} $. Then the quantum symmetry $Q_{\tau}^{Lin}$ generated by the matrix entries $q_{ij}$'s of $(q_{ij})_{3 \times 3}$ for $ C^*(\Gamma_{0})$  is isomorphic to $ (C(S^{1})*C(S^{1})*C(S^{1}),\Delta)$ but $\Gamma_{0} \notin \mathcal{G}$. \\
	
	\begin{center}
		\begin{figure}[htpb]  \label{Gamma0}
			\begin{tikzpicture}
			\draw[fill=black] (-2,0) circle (2pt) node[anchor=east]{$v_{1}$};
			\draw[fill=black] (0,0) circle (2pt) node[anchor=west]{$v_{2}$};
			\draw[fill=black] (2,0) circle (2pt) node[anchor=west]{$v_{3}$};
			
			\draw[black,thick](-2,0) --node{\rmidarrow} node[above]{$e_{12}$} (0,0);
			\draw[black,thick](0,0) edge[bend left= 60] node{\rmidarrow} node[above]{$e_{23}$} (2,0);
			\draw[black,thick](2,0) edge[bend left= 60] node{\lmidarrow} node[below]{$e_{32}$} (0,0);
			
			\end{tikzpicture}
			\caption{$\Gamma_{0}$}
		\end{figure}
	\end{center}
	For instance, $r^{-1}(s(e_{12})) = \emptyset $ and $r(e_{12})=s(e_{23})$. Therefore by Lemma \ref{l1}, $q_{e_{12}e_{23}}=0$. Applying the antipode $\kappa$ on both sides,  we get $q_{e_{23}e_{12}}=0$.  Again, $r(e_{23})=s(e_{32})$ implies $q_{e_{12}e_{32}}=0$. 
	For our convenience, we define $Q_{ef}=q_{ef}^{*}q_{ef}$. Since 
	\begin{align*}
		& ~ S_{e_{12}}^{*}S_{e_{12}}=S_{e_{32}}^{*}S_{e_{32}} \\
		\Rightarrow  &  \sum_{f \in E(\Gamma)}S_{f}^*S_{f}\otimes( Q_{f e_{12}}-Q_{f e_{32}})=0 \\
		\Rightarrow & ~ p_{v_2} \otimes ( Q_{e_{12} e_{12}}-Q_{e_{12} e_{32}}+ Q_{e_{32} e_{12}}-Q_{e_{32} e_{32}})+ p_{v_3} \otimes ( Q_{e_{23} e_{12}}-Q_{e_{23} e_{32}})=0\\
		\Rightarrow & ~ ( Q_{e_{23} e_{12}}-Q_{e_{23} e_{32}})=0 \\
		\Rightarrow & ~ Q_{e_{23} e_{32}}=0 
	\end{align*}
	Using $\kappa$ and combining all these, we have $q_{e_{12}e_{23}}=q_{e_{12}e_{32}}=q_{e_{23} e_{32}}=q_{e_{23}e_{12}}=q_{e_{32}e_{12}}=q_{e_{32} e_{23}}=0$. Therefore, $Q_{\tau}^{Lin} \cong (C(S^1)*C(S^1)*C(S^1)$, $\Delta)$.\\

	
	\item[(B)] Let $ P_{2,3} $ be a graph which is the disjoint union of $P_{2}$ (a simple directed path of length 1) and $P_{3}$ (a simple directed path of length 2). Then the adjacency matrix of $ P_{2,3} $ is given by $\begin{bmatrix}
	0 & 1 & 0 & 0 & 0\\
	0 & 0 & 0 & 0 & 0\\
	0 & 0 & 0 & 1 & 0\\
	0 & 0 & 0 & 0 & 1\\
	0 & 0 & 0 & 0 & 0\\
	\end{bmatrix} $. 
	In this case, also the universal object $Q_{\tau}^{Lin}$  remains $(C(S^{1})*C(S^{1})*C(S^{1}), \Delta)$ (upto isomorphism) though the underlying graph does not belong to $\mathcal{G}$, i.e. \textbf{(R2)} is also not necessary for Theorem \ref{T1}. This is again a simple application of Lemma \ref{l1} and Lemma \ref{l2} simultaneously. 
	\begin{center}
		\begin{figure}[htpb] \label{P23}
			\begin{tikzpicture}
			\draw[fill=black] (0,0) circle (2pt) node[anchor=north]{$v_{1}$};
			\draw[fill=black] (2,0) circle (2pt) node[anchor=north]{$v_{2}$};
			\draw[fill=black] (4,0) circle (2pt) node[anchor=north]{$v_{3}$};
			\draw[fill=black] (6,0) circle (2pt) node[anchor=north]{$v_{4}$};
			\draw[fill=black] (8,0) circle (2pt) node[anchor=north]{$v_{5}$};
			
			\draw[black,thick] (0,0)-- node{\rmidarrow}  node[above]{$e_{12}$} (2,0);
			\draw[black,thick] (4,0)-- node{\rmidarrow}  node[above]{$e_{34}$} (6,0);
			\draw[black,thick] (6,0)-- node{\rmidarrow}  node[above]{$e_{45}$} (8,0);
			\end{tikzpicture}
			\caption{$P_{2,3}$}
		\end{figure}
	\end{center} 
	
	\item[(C)] But \textbf{(R3)} is necessary for Theorem \ref{T1} because in counter example (3), we have shown that for any finite, connected, directed graph $ \Gamma $ containing at least two edges having the same source and range, universal object $Q_{\tau}^{Lin} \ncong \left(\mystar\limits_{|E(\Gamma)|} C(S^1),  \Delta\right)$. In other words, if for some graph $\Gamma$, universal object $Q_{\tau}^{Lin} \cong \left(\mystar\limits_{|E(\Gamma)|} C(S^1),  \Delta\right)$ then adjacency matrix of $\Gamma$ belongs to $M_{|V(\Gamma)|}(\{0,1\})$.  
	
\end{enumerate}


\section*{Acknowledgements}
The first author acknowledges the financial support from Department of Science and Technology, India (DST/INSPIRE/03/2021/001688). The second author also acknowledges the financial support from DST-FIST (File No. SR/FST/MS-I/2019/41).
We would also like to thank the anonymous referee for his/her useful comments on the older version of the paper.
\section*{Data availability}
Data sharing is not applicable to this article as no datasets were generated or analysed during the current study.

\section*{Declarations}
\textbf{Conflict of interest:} 
The authors have no conflicts to disclose.\\[2cm]

\raggedright{Ujjal Karmakar} \hfill                     
{Arnab Mandal}\\
{Presidency University} \hfill
{Presidency University}\\
{College Street, Kolkata-700073}  \hfill
{College Street, Kolkata-700073}\\
{West Bengal, India}  \hfill
{West Bengal, India}\\
{Email: \email{mathsujjal@gmail.com}} \hfill
{Email: \email{arnab.maths@presiuniv.ac.in}}

\end{document}